\documentclass[11pt]{article}
\usepackage{amsmath}
\usepackage{amssymb}
\usepackage{amsxtra}
\usepackage{amsthm}
\usepackage{mathrsfs}
\usepackage{mathtools}
\usepackage{graphicx}
\usepackage{epstopdf}
\usepackage{subfigure}
\usepackage{tabularx}
\usepackage{array}
\usepackage{float}
\usepackage{xcolor,graphicx,float}
\usepackage{caption}
\usepackage{authblk}
\usepackage[colorlinks,linkcolor=red,anchorcolor=blue,citecolor=blue]{hyperref}
\def\EquationsBySection{\def\theequation
	{\thesection.\arabic{equation}}
	\@addtoreset{equation}{section}}

\def \a{{\alpha}}
\def \b{{\beta}}
\def \G{{\Gamma}}

\usepackage{indentfirst}
\newtheorem{thm}{Theorem}[section]
\newtheorem{lem}{Lemma}[section]

\newtheorem{defn}{Definition}[section]
\newtheorem{rem}{Remark}[section]
\theoremstyle{definition}
\newtheorem{eg}{Example}[section]

\numberwithin{equation}{section}

\newcommand{\be}{\begin{equation}}
	\newcommand{\ee}{\end{equation}}

\newcommand{\bes}{\begin{equation*}}
	\newcommand{\ees}{\end{equation*}}

\renewcommand{\theequation}{\thesection.\arabic{equation}}

\newcommand{\bean}{\begin{eqnarray*}}
	\newcommand{\eean}{\end{eqnarray*}}

\newcommand\old[1]{}
\makeatother
\usepackage{appendix}
\usepackage{geometry}
\allowdisplaybreaks[3]
\EquationsBySection
\begin{document}
	\date{}
	\title{The Onsager-Machlup action functional for degenerate SDEs driven by fractional Brownian motion}
	\author{
		{\bf   Shanqi Liu $^{1}$ and Hongjun Gao $^{2}$}\thanks{The corresponding author. hjgao@seu.edu.cn}\\
		\footnotesize{ 1 School of Mathematical Science, Nanjing Normal University, Nanjing 210023, China}\\
		\footnotesize{ 2 School of Mathematics, Southeast University, Nanjing 211189, China}}
	\maketitle
	\noindent {\bf \small Abstract}{\small 
		\quad In this paper, the explicit expression of Onsager-Machlup action functional to degenerate stochastic differential equations driven by fractional Brownian motion is derived provided the diffusion coeffcient and reference path satisfy some suitable conditions. Then fractional Euler-Lagrange equations for Onsager-Machlup action functional are also obtained. Finally, some examples are provided to illustrate our results.}


	
	\section{Introduction}
Let $B^H={B_t^H, t\in[0,1]}$ be a fractional Brownian motion defined on the given complete filtered probability space $(\Omega,\mathcal{F},(\mathcal{F}_t)_{t\ge0},P)$ with Hurst index $0<H<1$. That is, $B^H$ a centered Gaussian process with covariance
$$E\big(B_t^HB_s^H\big)=\frac{1}{2}\big\{|t|^{2H}+|s|^{2H}-|t-s|^{2H}\big\}.$$
For $H=\frac{1}{2}$, fractional Brownian motion $B^H$ is classical Brownian motion, but for $H\neq\frac{1}{2}$ it is not a semimartingale nor a Markov process, and fractional Brownian motion is a very important stochastic processes in theory and application \cite{BH08,Mis08}.

In this paper, we study asymptotic behavior for solutions to the degenerate stochastic differential equation (DSDE) given by
\begin{equation}\label{dsde}
	\left\{\begin{array}{l}
		dX_t=\sigma(X_t, Y_t)dt, X_0=x,\\[3mm]
		dY_t=b(X_t,Y_t)dt+d B_t^H, Y_0=y,
	\end{array}\right.
\end{equation}
where, $t \in [0,1], x,y\in\mathbb{R}$, $\sigma\in C^1_b(\mathbb{R}\times\mathbb{R})$ and $b\in C^2_b(\mathbb{R}\times\mathbb{R})$.

More precisely, we are concerned with studying limiting behavior of ratios of the form
$$\gamma_{\varepsilon}(\phi)=\frac{P(\|Z-\phi\|_2\leq\varepsilon)}{P(\|B^H\|\leq\varepsilon)}$$
when $\varepsilon$ tends to zero, where $Z:=(X,Y)$ is the solution process of DSDE $(\ref{dsde})$, $\phi:=(\phi^{(1)},\phi^{(2)})$ is the (deterministic) reference path satisfying some regularity conditions and appropriate structure (see \eqref{structure of reference path} below), $\|Z-\phi\|_2:=\sqrt{\|X-\phi^{(1)}\|^2+\|Y-\phi^{(2)}\|^2}$ and $\|\cdot\|$ is a suitable norm defined on the functions from $[0,1]$ to proper Euclidean space. When $$\lim\limits_{\varepsilon\rightarrow 0}\gamma_{\varepsilon}(\phi)=\exp (J_0(\phi))$$ for all $\phi$ in a reasonable class of functions, then if the limit is exists, the functional $J_0$ is called the Onsager-Machlup (OM) action functional associated to $(\ref{dsde})$.

OM action functional was first given by Onsager and Machlup \cite{OM1, OM2} as the probability density functional for diffusion processes with linear drift and constant diffusion coefficients, and then Tizsa and Manning \cite{TM57} generalized the results of Onsager and Machlup to nonlinear equations. The key point was to express the transition probability of a diffusion process by a functional integral over paths of the process, and the integrand was called OM action function. Then regarding OM action function as a Lagrangian,  the most probable path of the diffusion process was determined by variational principle. However, the paths of the diffusion process are almost surely nowhere differentiable. This means that it is not feasible to use the variational principle for the path of diffusion process. To modify that, Stratonovich \cite{Str57} proposed the rigorous mathematical idea: one can ask for the probability that a path lies within a certain region, which may be a tube along a differentiable function (mostly called as reference path), comparing the probabilities of different tubes of the same 'thicknes', the OM action function is expressed by reference path instead of the path of diffusion process.

Indeed, classical Onsager-Machlup theory as shown in \cite{DB78,IW14}, for any reference path $\phi\in C^2([0,1], \mathbb{R}^m)$, the Onsager-Machlup action functional for $X_t$ is defined by
$$\lim_{\delta\to 0}\frac{P\big(\sup\limits_{t \in[0, 1]}\big|X_t-\phi_t\big| \leq \varepsilon\big)}{P(\sup\limits_{t \in[0, 1]}|W_t|\leq\varepsilon)}=\exp\Big(\frac{1}{\delta^2}\big\{L_\delta(\phi,\dot{\phi})\big\}\Big),$$
where $X_t$ is the solution of non-degenerate stochastic differential equations (SDEs):
\begin{align}
	d X_t=f(t,X_t)d t+\delta d W_t, X_0=x,\nonumber
\end{align}
where $f\in C_b^2([0,1]\times \mathbb{R}^{m})$, $W(t)$ is a $m$-dimensional Brownian motion and exact expression of Onsager-Machlup action functional is given by
$$L_\delta(\phi,\dot{\phi})=-\frac{1}{2}\int_{0}^{1}\left|\dot{\phi}_t-f(t,\phi_t)\right|^{2}d t-\frac{\delta}{2}\int_{0}^{1} \operatorname{div}_{x}f(t,\phi_t)d t,$$	
where $\operatorname{div}_{x}$  denote the divergence on the $\phi_t\in\mathbb{R}^m$. As noise intensity parameter $\delta \to 0$,
Onsager-Machlup action functional coincides with Freidlin-Wentzell action functional \cite{FW84} formally.

From the beginning of irreplaceable contributions of Stratonovich \cite{Str57}, OM action functional theory was starting to receive considerable attention by  mathematicians. Many different approaches and new problems have arisen in this process. We first review works on OM action functional for stochastic differential equations driven by non-degenerate noise \cite{Cap95,Cap20,DB78, DZ91,FK82,HT96,HT96,IW14,MN02,SZ92}.  Ikeda and Watanabe \cite{IW14} derived the OM action functional for reference path $\phi\in C^2([0,1],\mathbb{R}^{d})$ and taking the supremum norm $\|.\|_{\infty}$. D\"{u}rr and Bach \cite{DB78} obtained the same results based on the Girsanov transformation of the quasi-translation invariant measure and the potential function (path integral representation). Shepp and Zeitouni \cite{SZ92} proved that this limit theorem still holds for every norm equivalent to the supremum norm and $\phi-x$ in the Cameron-Martin space. Capitaine \cite{Cap95}  extended this result to a large class of natural norms on the Wiener space including particular cases of H\"{o}lder, Sobolev, and Besov norms. Hara and Takahashi provided in \cite{HT96} a computation of the OM action functional of an elliptic diffusion process for the supremum norm and this result was extended in \cite{Cap20} by Capitaine to norms that dominate supremum norm. In particular, the norms $\|\cdot\|$ could be any Euclidean norm dominating $L^2$-norm in the case of $\mathbb{R}^d$.

In addition to the derivation of OM action functional for SDEs derived by Brownian motion, the derivation of OM action functional for SDEs derived by fractional Brownian motion has also attracted much attention. Moret and Nualart \cite{MN02} first obtained the OM action functional for SDEs derived by fractional Brownian motion in singular and regular cases respectively. That is, for the following SDEs:
\begin{equation}\label{Nsde}
	\left\{\begin{array}{l}
		dX_t=b(X_t)dt+dB_t^H,\\[3mm]
		X_0=x\in\mathbb{R}.
	\end{array}\right.
\end{equation}
In $\frac{1}{4}<H<\frac{1}{2}$ (singular case) $\|\cdot\|$ can be either the
supremum norm or $\alpha<H-\frac{1}{ 4}$ H\"{o}lder norm. In $H>\frac{1}{2}$ (regular case) the H\"{o}lder norm can only be taken as $H-\frac{1}{2}<\alpha< H-\frac{1}{4}$. The accurate expression of OM action functional for \eqref{Nsde} obtained in both cases is:
$$L(\phi,\dot{\phi})=-\frac{1}{2}\int_{0}^{1}\Big(\dot{\phi}_s-\big(K^H\big)^{-1}\int_{0}^{s}b(\phi_u)du\Big)^2ds-\frac{1}{2}d_H\int_{0}^{1}b'(\phi_s)ds,$$
where $\dot{\phi}$ is the function such that $K^H \dot{\phi}=\phi-x$, $d_H$ is a constant depending on $H$ and the definition of $K^H$ see Definition \ref{K^H}.

Then inspired by Bardina, Rovira
and Tindel's work \cite{BRT1} of OM action functional for stochastic evolution equations and \cite{MN02}, Liang \cite{Lia10} studyed conditional exponential moment by Karhunen-{L}o\'{e}ve expansion for stochastic convolution of cylindrical fractional {B}rownian motions, but it is still an open problem about deriving the OM functional of the stochastic evolution equation driven by fractional Brownian motion
at present.

On the other hand, OM action functional for degenerate SDEs has also attracted a lot of interest. In order to determine the OM action functional for degenerate SDEs. Kurchan \cite{Kur98} derived OM action functional via a Fokker-Planck equation \cite{Ris89} corresponding to the Langevin equation. Taniguchi and Cohen \cite{TC07, TC08} obtained the OM action functional for the Langevin equation by the path integral approach. A rigorous mathematical treatment of this problem was initiated by \cite{AB99} and \cite{CN95} independently. Chaleyat-Maurel and Nualart \cite{CN95} derived Onsager-Machlup action functional for second-order stochastic differential equations with two-point boundary value condition. To derive the maximum likelihood state estimator for degenerate SDEs, Aihara and Bagchi \cite{AB99} extended OM action functional into a degenerate version of OM action functional for reference path $\phi \in H^1([0,1],\mathbb{R}^d)$ with supremum norm by the approach of \cite{SZ92}.

Compared with the existing results in this direction, the main innovation of this paper consists in that we first derived OM action functional for degenerate stochastic differential equations driven by fractional Brownian motion. This result is obtained using the ordinary approach with the following ingredients:

$\bullet$
The application of Girsanov theorem which involves the operator $(K^H)^{-1}$ and some results associated with conditional exponential moments and small ball probabilities

$\bullet$
A suitable structure of reference path under degenerate noise

$\bullet$
The equivalence between two different small ball probabilities

$\bullet$ A lot of accurate estimation. For example, we need to deal with conditional exponential moment of the stochastic integral $\int_{0}^{1}s^{-\alpha}\big(I^{\alpha}_{0+}u^{\alpha}b(\phi^{(1)}_u+\kappa_u,\phi^{(2)}_u+B_u^H)\big)(s)dW_s$ in the singular case carefully

Our main result Theorem \ref{th:singular case} and \ref{th:regular case} provide asymptotic behaviour of the small
ball probabilities of the solution to \eqref{dsde} and reference path. As a consequence of the above result we are able to obtain Euler-Lagrange fractional equations for Onsager-Machlup action functional, which
provides a characterization of the most probable path of the solution process \eqref{dsde},  see Theorem
\ref{EL} for the precise statement.

Before starting our results we give some notations.
\begin{defn}\label{K^H}
	The fractional Brownian motion has the integral representation in law:
	\begin{align}\label{relation between B and W}
		B_t^H=\int_{0}^{1}K^H(t,s)dW_s,
	\end{align}
	where $W$ is a standard Brownian motion and $K^H$ is the square integral kernel:
	\begin{align}\label{K^H(r,u)}
		K^H(r,u)=c_H(r-u)^{H-\frac{1}{2}}+c_H(\frac{1}{2}-H)\int_{u}^{r}(\theta-u)^{H-\frac{3}{2}}\Big(1-(\frac{u}{\theta})^{\frac{1}{2}-H}\Big)d\theta,
	\end{align}
	with
	$$c_H=\Big(\frac{2H\G(\frac{3}{2}-H)}{\G(H+\frac{1}{2})\G(2-2H)}\Big)^{\frac{1}{2}}.$$
\end{defn}
We also denote by $K^H$ the operator in $L^2([0,1])$ associated with the kernel $K^H$, that is
$$\big(K^Hh\big)(s)=\int_{0}^{s}K^H(s,r)h(r)dr.$$
If $p>\frac{1}{H+\frac{1}{2}}$ the operator $K^H$ is continuous in $L^p$, and we denote by
$$\mathcal{H}^p=\{K^H h, h\in L^p([0,1])\}$$
the image of $L^p([0,1])$ by $K^H$.
For $f\in L^1[a,b]$ and $\a>0$ the right-side fractional Riemann-Liouville integrals of $f$ of order $\a$ on $(a,b)$ are defined at almost all $x$ by
$$(I^{\a}_{a^+}f)(x)=\frac{1}{\G(\a)}\int_{a}^{x}(x-y)^{\a-1}f(y)dy,$$
where $\G$ denotes the Euler function.

The fractional derivative can be introduced as inverse operation. If $1\leq p<\infty$, we denote by $I^{\a}_{a^+}(L^p)$ the image of $L^p([a,b])$ by the operator $I^{\a}_{a^+}$. If $f\in I^{\a}_{a^+}(L^p)$, the function $\phi$ such that $f=I^{\a}_{a^+}\phi$ is unique in $L^p$ and it agrees with the left-sided Riemann-Liouville derivative of $f$ of order $\a$ defined by
$$(D^{\a}_{a^{+}}f)(x)=\frac{1}{\G(1-\a)}\frac{d}{dx}\int_{a}^{x}\frac{f(y)}{(x-y)^{\a}}dy.$$
\subsection{Main Results}
\label{sec:model}

We are now ready to state the main results of this paper.

\begin{thm}
	\label{th:singular case}
	Let $Z$ be the solution of \eqref{dsde} with Hurst index $\frac{1}{4}<H<\frac{1}{2}$. Let reference path $\phi=(\phi^{(1)},\phi^{(2)})$ be a function such that $\phi^{(2)}-y\in \mathcal{H}^p$ with $p>\frac{1}{H}$. Assume $\sigma\in C^1_b(\mathbb{R}^2)$ and $b\in C^2_b(\mathbb{R}^2)$. Then the Onsager-Machlup action functional of $Z$ for the norms $\|\cdot\|_{\b}$  with $0<\beta<H-\frac{1}{4}$ and $\|\cdot\|_{\infty}$ exists and is given by
	$$L(\phi,\dot{\phi})=-\frac{1}{2}\int_{0}^{1}\big|\dot{\phi}^{(2)}_s-s^{-\a}\big(I^{\a}_{0+}u^{\a}b(\phi_u)\big)(s)\big|^2ds-\frac{d_H}{2}\int_{0}^{1}b_y(\phi_{s})ds,$$
	where
	$$d_H=\Bigg(\frac{2H\G(\frac{3}{2}-H)\G(H+\frac{1}{2})}{\G(2-2H)}\Bigg)^{\frac{1}{2}},$$
	$\a=\frac{1}{2}-H$	and $\phi$ is the function such that
	\begin{equation}
		\left\{\begin{array}{l}
			\phi^{(1)}_t=x+\int_{0}^{t}\sigma(\phi)(s)ds,\\ \phi^{(2)}_t=y+(K^H\dot{\phi}^{(2)})(t).\nonumber
		\end{array}\right.
	\end{equation}
\end{thm}

\begin{thm}
	\label{th:regular case}
	Let $Z$ be the solution of \eqref{dsde} with Hurst index $\frac{1}{2}<H<1$. Let reference path $\phi=(\phi^{(1)},\phi^{(2)})$ be a function such that $\phi^{(2)}-y\in \mathcal{H}^2$. Assume $\sigma\in C^1_b(\mathbb{R}^2)$ and $b\in C^3_b(\mathbb{R}^2)$. Then the Onsager-Machlup action functional of $Z$ for the norms $\|\cdot\|_{\b}$  with $H-\frac{1}{2}<\beta<H-\frac{1}{4}$ exists and is given by
	$$L(\phi,\dot{\phi})=-\frac{1}{2}\int_{0}^{1}\big|\dot{\phi}^{(2)}_s-s^{\a}\big(D^{\a}_{0+}u^{-\a}b(\phi_u)\big)(s)\big|^2ds-\frac{d_H}{2}\int_{0}^{1}b_y(\phi_{s})ds,$$
	where
	$$d_H=\Bigg(\frac{2H\G(\frac{3}{2}-H)\G(H+\frac{1}{2})}{\G(2-2H)}\Bigg)^{\frac{1}{2}},$$
	$\a=H-\frac{1}{2}$ and $\phi$ is the function such that
	\begin{equation}
		\left\{\begin{array}{l}
			\phi^{(1)}_t=x+\int_{0}^{t}\sigma(\phi)(s)ds,\\ \phi^{(2)}_t=y+(K^H\dot{\phi}^{(2)})(t).\nonumber
		\end{array}\right.
	\end{equation}
\end{thm}
\begin{rem}
	We do not need to impose any condition on $\phi^{(1)}$. Since $\sigma\in C^1_b(\mathbb{R})$ and the structure of $\phi^{(1)}_t=x+\int_{0}^{t}\sigma(\phi)(s)ds$, so $\phi^{(1)}$ is a ''good'' function.
\end{rem}
\begin{rem}\label{app version}
	As an important application of OM action functional theory, the most probable path is achieved by applying Euler-Lagrange equation to OM action functional. However, it is not appropriate to directly apply Theorem \ref{th:singular case} and \ref{th:regular case} to some specific examples. For example $\sigma(x,y)=y$ is not satisfied the conditions of Theorem \ref{th:singular case} and \ref{th:regular case} about the uniform boundedness of $\sigma$ and $b$. Therefore, in order to include the above example, we assume reference path $\phi \in {C}_b^2([0,T],\mathbb{R}\times \mathbb{R})$ and assume $\sigma\in C^2,b\in C^1$ only. Which implies that \eqref{dsde} admits a local solution. So instead of assume $T=1$ as previous setting. Here we assume $T$ is a suitable time such that the existence and uniqueness of the solution of \eqref{dsde} can be guaranteed. By similar procedure as the proof of Theorem \ref{th:singular case} and \ref{th:regular case}, we can get similar results but second-order stochastic differential equation is included.
\end{rem}
Next, we would apply the idea of variational principle \cite{AT09} to OM functional and obtain fractional Euler-Lagrange equations in non-degenerate and a class of degenerate cases ($\sigma(x,y)=y$ and $\dot{\phi}^{(1)}=\phi^{(2)}$). More precisely, let us consider the following minimization problems:
\begin{align}\label{non-degenerate-1 optimal}
	\min\limits_{\phi\in{C}_b^2([0,T],\mathbb{R})} I(\phi)=\frac{1}{2}\int_{0}^{T}\big|\dot{\phi}_s-s^{\a}\big(I^{\a}_{0+}u^{-\a}b(\phi_u)\big)(s)\big|^2ds+\frac{d_H}{2}\int_{0}^{T}b'(\phi_{s})ds,
\end{align}
where $\frac{1}{4}<H<\frac{1}{2}$ in non-degenerate case.
\begin{align}\label{non-degenerate-2 optimal}
	\min\limits_{\phi\in{C}_b^2([0,T],\mathbb{R})} I(\phi)=\frac{1}{2}\int_{0}^{T}\big|\dot{\phi}_s-s^{\a}\big(D^{\a}_{0+}u^{-\a}b(\phi_u)\big)(s)\big|^2ds+\frac{d_H}{2}\int_{0}^{T}b'(\phi_{s})ds,
\end{align}
where $\frac{1}{2}<H<1$ in non-degenerate case.
\begin{align}\label{degenerate optimal-1}
	\min\limits_{\phi\in{C}_b^4([0,T],\mathbb{R})} I(\phi^{(1)})=\frac{1}{2}\int_{0}^{T}\big|\ddot{\phi}^{(1)}_s-s^{\a}\big(I^{\a}_{0+}u^{-\a}b(\phi^{(1)}_u,\dot{\phi}^{(1)}_u)\big)(s)\big|^2ds+\frac{d_H}{2}\int_{0}^{T}b_y(\phi^{(1)}_s,\dot{\phi}^{(1)}_{s})ds,
\end{align}
where $\frac{1}{4}<H<\frac{1}{2}$ in degenerate case.
\begin{align}\label{degenerate optimal-2}
	\min\limits_{\phi\in{C}_b^4([0,T],\mathbb{R})} I(\phi^{(1)})=\frac{1}{2}\int_{0}^{T}\big|\ddot{\phi}^{(1)}_s-s^{\a}\big(D^{\a}_{0+}u^{-\a}b(\phi^{(1)}_u,\dot{\phi}^{(1)}_u)\big)(s)\big|^2ds+\frac{d_H}{2}\int_{0}^{T}b_y(\phi^{(1)}_s,\dot{\phi}^{(1)}_{s})ds,
\end{align}
where $\frac{1}{2}<H<1$ in degenerate case.
\begin{thm}\label{EL}
	Let $\phi(\cdot)$ and $\phi^{(1)}(\cdot)$ be a local minimizer of  problem \eqref{non-degenerate-1 optimal}, \eqref{non-degenerate-2 optimal} and \eqref{degenerate optimal-1}, \eqref{degenerate optimal-2}, respectively. Then,  $\phi(\cdot)$ and $\phi^{(1)}(\cdot)$ satisfy the fractional Euler–Lagrange equations for corresponding cases:
	
	$\bullet$ \bf{Non-degenerate diffusion: singular case}
	$$I^{\a}_{1-}\big(u^{-2\a}(I^{\a}_{0+})v^{\a}b(\phi_v)(u)\big)s^{\a}b'(\phi_s)+\frac{d_H}{2}b''(\phi_{s})=\frac{d}{dt}\big(\dot{\phi}_s-s^{-\a}\big(I^{\a}_{0+}u^{\a}b(\phi_u)\big)(s)\big).$$
	$\bullet$ \bf{Non-degenerate diffusion: regular case}
	$$D^{\a}_{1-}\big(u^{-2\a}(D^{\a}_{0+})v^{\a}b(\phi_v)(u)\big)s^{\a}b'(\phi_s)+\frac{d_H}{2}b''(\phi_{s})=\frac{d}{dt}\big(\dot{\phi}_s-s^{-\a}\big(D^{\a}_{0+}u^{\a}b(\phi_u)\big)(s)\big).$$
	$\bullet$ \bf{Degenerate diffusion: singular case}
	\begin{align}
		0&=\frac{d^2}{dt^2}\big(\ddot{\phi}^{(1)}_s-s^{\a}b_x(\phi_s^{(1)},\dot{\phi}^{(1)}_s)\Big(I^{\a}_{1-}u^{\a}\Big(\ddot{\phi}^{(1)}_u-u^{-\a}\big(I^{\a}_{0+}v^{\a}b(\phi^{(1)}_v,\dot{\phi}^{(1)}_v\big)(u)\Big)\Big)(s)\nonumber\\&+\frac{d}{dt}\Big[s^{\a}b_y(\phi_s^{(1)},\dot{\phi}^{(1)}_s)\Big(I^{\a}_{1-}u^{\a}\Big(\ddot{\phi}^{(1)}_u-u^{-\a}\big(I^{\a}_{0+}v^{\a}b(\phi^{(1)}_v,\dot{\phi}^{(1)}_v\big)(u)\Big)\Big)(s)\Big]\nonumber\\&+\frac{d_H}{2}b_{yx}(\phi^{(1)}_{s},\dot{\phi}^{(1)}_s)-\frac{d_H}{2}\frac{d}{dt}\Big(b_{yy}(\phi^{(1)}_{s},\dot{\phi}^{(1)}_s)\Big).\nonumber
	\end{align}
	$\bullet$ \bf{Degenerate diffusion: regular case}
	\begin{align}
		0&=\frac{d^2}{dt^2}\big(\ddot{\phi}^{(1)}_s-s^{\a}b_x(\phi_s^{(1)},\dot{\phi}^{(1)}_s)\Big(D^{\a}_{1-}u^{\a}\Big(\ddot{\phi}^{(1)}_u-u^{-\a}\big(D^{\a}_{0+}v^{\a}b(\phi^{(1)}_v,\dot{\phi}^{(1)}_v\big)(u)\Big)\Big)(s)\nonumber\\&+\frac{d}{dt}\Big[s^{\a}b_y(\phi_s^{(1)},\dot{\phi}^{(1)}_s)\Big(D^{\a}_{1-}u^{\a}\Big(\ddot{\phi}^{(1)}_u-u^{-\a}\big(D^{\a}_{0+}v^{\a}b(\phi^{(1)}_v,\dot{\phi}^{(1)}_v\big)(u)\Big)\Big)(s)\Big]\nonumber\\&+\frac{d_H}{2}b_{yx}(\phi^{(1)}_{s},\dot{\phi}^{(1)}_s)-\frac{d_H}{2}\frac{d}{dt}\Big(b_{yy}(\phi^{(1)}_{s},\dot{\phi}^{(1)}_s)\Big).\nonumber
	\end{align}
\end{thm}
\subsection{Examples}
We now compare the results of non-degenerate case and degenerate case by two examples. We choose the norm is $\|\cdot\|_{\infty}$ and $H<\frac{1}{2}$.
\begin{eg}
	Consider the following scalar SDE driven by fractional Brownian motion:
	\begin{align}\label{example1}
		d X_t=[X_t-X_t^3]d t+d B^H_t, X_0=1,
	\end{align}
	by Theorem $7$ in \cite{MN02} and Remark \ref{app version} we can obtain the Onsager-Machlup action functional for \eqref{example1}:
	$$L(\phi,\dot{\phi})=-\frac{1}{2}\int_{0}^{T}\big|\dot{\phi}_s-s^{\a}\big(I^{\a}_{0+}u^{-\a}(\phi_u-\phi_u^3)\big)(s)\big|^2ds-\frac{d_H}{2}\int_{0}^{T}(1-3\phi^2_{s})ds,$$
	where $\phi_t=1+(K^H\dot{\phi})(t)$ and corresponding fractional Euler-Lagrange equation:
	$$I^{\a}_{1-}\big(u^{-2\a}(I^{\a}_{0+})v^{\a}(\phi_v-\phi^3_v)(u)\big)s^{\a}(1-3\phi^2_s)-3d_H\phi_{s}=\frac{d}{dt}\big(\dot{\phi}_s-s^{-\a}\big(I^{\a}_{0+}u^{\a}(\phi_u-\phi_u^3)\big)(s)\big).$$
\end{eg}
\begin{eg}
	We take the stochastic scalar system as follow:
	\begin{align}\label{ex1}
		\ddot{X}=X-X^3+\dot{B}^H,X_0=-1,\dot{X}_0=1,
	\end{align}
	which can be rewriten as a system of first-order SDEs:
	\begin{align}\label{ex11}
		\begin{aligned}
			d\left(\begin{array}{c}
				X_t \\
				\dot{X}_t
			\end{array}\right)=\left(\begin{array}{c}
				\dot{X}_t\\
				X_t-X_t^3
			\end{array}\right) dt+\left(\begin{array}{c}
				0 \\
				1
			\end{array}\right) d B_t^H.
		\end{aligned}
	\end{align}
	Defining $\mathbf{X}_t=\left(\begin{array}{c}
		X_t \\
		\dot{X}_t\end{array}\right)$ and by Theorem \ref{th:singular case} and Remark \ref{app version} we can obtain the Onsager-Machlup action functional for \eqref{ex1} or \eqref{ex11}:
	$$L(\phi,\dot{\phi})=-\frac{1}{2}\int_{0}^{T}\big|\ddot{\phi}^{(1)}_s-s^{\a}\big(I^{\a}_{0+}u^{-\a}(\phi^{(1)}_u-(\phi^{(1)}_u)^3)\big)(s)\big|^2ds,$$
	where $\phi^{(1)}_t=-1+\int_{0}^{t}\phi_s^{(2)}ds, \phi^{(2)}_t=1+(K^H\dot{\phi}^{(2)})(t)$ and corresponding fractional Euler-Lagrange equation:
	\begin{align}
		\frac{d^2}{dt^2}\big(\ddot{\phi}^{(1)}_s-s^{\a}(1-3(\phi_s^{(1)})^2)\Big(I^{\a}_{1-}u^{\a}\Big(\ddot{\phi}^{(1)}_u-u^{-\a}\big(I^{\a}_{0+}v^{\a}(\phi^{(1)}_v-(\phi^{(1)}_v-(\phi_v^{(1)})^3)\big)(u)\Big)\Big)(s)=0.\nonumber
	\end{align}	
\end{eg}
{\em Convention on constants:} Throughout the paper C denotes a
positive constant whose value may change from line to line. The dependence of
constants on parameters when relevant will be denoted by special
symbols or by mentioning the parameters in brackets, for e.g.
$C(\a,
\b)$.

\section{Setting}\label{sec:opmr}

In this section, we recall some classical results for fractional calculus, and we introduce the structure of reference path under degenerate noise. Keeping this structure in mind, we further convert the problem of deriving Onsager-Machlup action functional into more clearer conditional exponential moments by Girsanov's Theorem. Finally, some key lemmas and propositions are given.

\subsection{Fractional calculus}
For $f\in L^1[a,b]$ and $\a>0$ the right-side fractional Riemann-Liouville integrals of $f$ of order $\a$ on $(a,b)$ are defined at all $x$ by
$$(I^{\a}_{a^+}f)(x)=\frac{1}{\G(\a)}\int_{a}^{x}(x-y)^{\a-1}f(y)dy,$$
where $\G$ denotes the Euler function.

This integral extends the usual $n$-order iterated integrals of $f$ for $\a=n\in \mathbb{N}$.
We have the first composition formula
$$I^{\a}_{a^+}(I^{\beta}_{a+}f)=I^{\a+\beta}_{a^+}f.$$

The fractional derivative can be introduced as inverse operation. If $1\leq p<\infty$, we denote by $I^{\a}_{a^+}(L^p)$ the image of $L^p([a,b])$ by the operator $I^{\a}_{a^+}$. If $f\in I^{\a}_{a^+}(L^p)$, the function $\phi$ such that $f=I^{\a}_{a^+}\phi$ is unique in $L^p$ and it agrees with the left-sided Riemann-Liouville derivative of $f$ of order $\a$ defined by
$$(D^{\a}_{a^{+}}f)(x)=\frac{1}{\G(1-\a)}\frac{d}{dx}\int_{a}^{x}\frac{f(y)}{(x-y)^{\a}}dy.$$
When $\a p>1$ any function in $I^{\a}_{a^+}(L^p)$ is $\big(\a-\frac{1}{p}\big)$-H$\mathrm{\ddot{o}}$lder continuous. On the other hand, any H$\mathrm{\ddot{o}}$lder continuous function of order $\b>\a$ has fractional derivative of order $\a$. The derivative of $f$ has the following Weyl representation:
\begin{align}\label{weyl representation}
	(D^{\a}_{a^+}f)(x)&=\frac{1}{\G(1-\a)}\Big(\frac{f(x)}{(x-a)^{\a}}+\a\int_{a}^{x}\frac{f(x)-f(y)}{(x-y)^{\a+1}}dy\Big)\mathbb{I}_{(a,b)}(x),
\end{align}
where the convergence e of the integrals at the singularity $x=y$ holds in $L^p$-sense.

Recall that by construction for $f\in I^{\a}_{a^+}(L^p)$,
$$I^{\a}_{a^+}(D^{\a}_{a^+}f)=f$$
and for general $f\in L^1([a,b])$ we have
$$D^{\a}_{a^+}(I^{\a}_{a^+}f)=f.$$
If $f\in I^{\a+\b}_{a^+}(L^1), \a\geq0,\b\geq0,\a+\b\leq1$ we have the second composition formula
$$D^{\a}_{a^+}(D^{\beta}_{a+}f)=D^{\a+\beta}_{a^+}f.$$
The following estimate for the norm of the fractional integral will be used later in this paper,
\begin{align}\label{fractional integral estimate}
	\|I^{\a}_{a^+}f\|_{L^p([a,b])}\leq\frac{(b-a)^{\a}}{\a|\G(\a)|}\|f\|_{L^p([a,b])},
\end{align}
provided $f\in L^p([a,b])$.
\subsection{The structure of reference path under degenerate noise}
Let $B^H={B_t^H, t\in[0,1]}$ be a fractional Brownian motion with Hurst index $0<H<1$ ($H\neq\frac{1}{2}$) defined on the given complete filtered probability space $(\Omega,\mathcal{F},(\mathcal{F}_t)_{t\ge0},P)$. Consider the degenerate stochastic differential equation:
\begin{equation}
	\left\{\begin{array}{l}
		dX_t=\sigma(X_t, Y_t)dt, X_0=x,\\ dY_t=b(X_t,Y_t)dt+d B_t^H, Y_0=y,\nonumber
	\end{array}\right.
\end{equation}
where we assume $\sigma\in C_b^1(\mathbb{R}^2, \mathbb{R})$, $b\in C_b^2(\mathbb{R}^2, \mathbb{R})$ and we define $Z_t=(X_t,Y_t)$.

We will denote reference path $\phi=(\phi^{(1)},\phi^{(2)})$ the function in $L^p([0,1], \mathbb{R}^2)$ such that
\begin{equation}\label{structure of reference path}
	\left\{\begin{array}{l}
		\phi^{(1)}_t=x+\int_{0}^{t}\sigma(\phi)(s)ds,\\ \phi^{(2)}_t=y+(K^H\dot{\phi}^{(2)})(t).
	\end{array}\right.
\end{equation}
Our motivation to construct the structure of reference path, on the one hand, it is inspired by the derivation of OM action functional for degenerate SDEs with Brownian motion (see \cite{AB99} Theorem $2$). On the other hand, it comes from the deviation of OM action functional for the fractional Brownian motion (see \cite{MN02} eq. $(9)$).

After we have completed the construction of the structure of reference path, by the equivalence between two different small ball probabilities (see Lemma \ref{equivalence}), we will rewrite the ratios as
$$\gamma_{\varepsilon}(\phi)=\frac{P(\|Z-\phi\|_2\leq\varepsilon)}{P(\|B^H\|\leq\varepsilon)}=\frac{P(\|Y-\phi^{(2)}\|\leq\varepsilon)}{P(\|B^H\|\leq\varepsilon)}.$$
Throughout the paper, we will use the more convenient ratio at the right hand of the above equality.
\subsection{Application of Girsanov's Theorem}
Consider the following auxiliary degenerate SDE on $\mathbb{R}\times \mathbb{R}$:
\begin{equation}
	\left\{\begin{array}{l}
		d\tilde{X}_t=\sigma(\tilde{X}_t,\tilde{Y}_t)  d t, \\
		d\tilde{Y}_t=\dot{\phi}^{(2)}_t dt+d B^H_t.\nonumber
	\end{array}\right.
\end{equation}
Letting $\tilde{Z}_t=(\tilde{X}_t,\tilde{Y}_t), (\tilde{X}_0,\tilde{Y}_0)=(x,y) $ and taking
$$
\begin{aligned}
	\begin{aligned}
		&\tilde{W}_t=W_t-\int_{0}^{t} \eta_s ds, \\
		&\eta_s=\Big((K^H)^{-1}\Big(b(\tilde{X}_u,\tilde{Y}_u)\Big)(s)-\dot{\phi}^{(2)}_s\Big), \quad s, t \in[0, 1],\\&\tilde{B}_t^H=\int_{0}^{t}K^H(t,s)d\tilde{W}_s,
	\end{aligned}
\end{aligned}
$$
where $W_t$ is given in \eqref{relation between B and W}. Applying classical Girsanov's theorem we have that $\tilde{W}$ is a standard Brownian motion under the probability measure $\tilde{P}$ defined by
\begin{align}\label{eta}
	\frac{d\tilde{P}}{dP}=\exp\Big(\int_{0}^{1}\eta_sds-\frac{1}{2}\int_{0}^{1}\eta_s^2ds\Big).
\end{align}
Meanwhile, the application of Girsanov's theorem requires the process $\eta$ to be adapted and $E(\frac{d\tilde{P}}{dP})=1$. We will prove that $\eta$ satisfies these conditions in Lemma \ref{girsanov}. Then under the probability measure $\tilde{P}$, $\tilde{B}^H$ is a fractional Brownian motion, and under transformed probability space $(\Omega,\mathcal{F},\tilde{P})$ $(\tilde{X},\tilde{Y},\tilde{W},\tilde{B}^H)$ is the solution of the following SDE:
\begin{equation}
	\left\{\begin{array}{l}
		d\tilde{X}_t=\sigma(\tilde{X}_t, \tilde{Y}_t)dt, \quad \tilde{X}_0=x\in \mathbb{R},\\
		d\tilde{Y}_t=b(\tilde{X}_t,\tilde{Y}_t)d t+d\tilde{B}_t^H,\quad  \tilde{Y}_0=y\in \mathbb{R}.\nonumber
	\end{array}\right.
\end{equation}
So we could reduce small ball probability as
\begin{align}\label{all}
	&P(\|Y-\phi^{(2)}\|\leq\varepsilon)\nonumber\\&=\tilde{P}(\|\tilde{Y}-\phi^{(2)}\|\leq \varepsilon)=E\Big(\frac{d\tilde{P}}{dP}\mathbb{I}_{\|B^H\|\leq\varepsilon})\nonumber\\&=E\Big(\exp\Big(\int_{0}^{1}\eta_sds-\frac{1}{2}\int_{0}^{1}\eta_s^2ds\Big)\mathbb{I}_{\|B^H\|\leq\varepsilon}\Big)\nonumber\\&=E\Big(\exp\Big(\int_{0}^{1}\Big((K^H)^{-1}\Big(b(\tilde{X}_u,\tilde{Y}_u)\Big)(s)-\dot{\phi}^{(2)}_s\Big)ds-\frac{1}{2}\int_{0}^{1}\Big((K^H)^{-1}\Big(b(\tilde{X}_u,\tilde{Y}_u)\Big)(s)-\dot{\phi}^{(2)}_s\Big)^2ds\Big)\mathbb{I}_{\|B^H\|\leq\varepsilon}\Big)\nonumber\\&=E\Big(\exp\Big(\int_{0}^{1}\big((K^H)^{-1}b(\tilde{X}_u,\phi^{(2)}_u+B_u^H)\big)(s)dW_s+\int_{0}^{1}(-\dot{\phi}^{(2)}_s)dW_s+\frac{1}{2}\int_{0}^{1}\big|\dot{\phi}^{(2)}_s-\big((K^H)^{-1}b(\phi_u)\big)(s)\big|^2ds\nonumber\\&-\frac{1}{2}\int_{0}^{1}\big|\dot{\phi}^{(2)}_s-\big((K^H)^{-1}b(\phi_u)\big)(s)\big|^2ds-\frac{1}{2}\int_{0}^{1}\Big((K^H)^{-1}\Big(b(\tilde{X}_u,\tilde{Y}_u)\Big)(s)-\dot{\phi}^{(2)}_u\Big)^2ds\Big)\mathbb{I}_{\|B^H\|\leq\varepsilon}\Big)\nonumber\\&=E\Big(\exp(I_1+I_2+I_3+I_4)\mathbb{I}_{\|B^H\|\leq\varepsilon}\Big)\times\exp\Big(-\frac{1}{2}\int_{0}^{1}\big|\dot{\phi}^{(2)}_s-\big((K^H)^{-1}b(\phi_u)\big)(s)\big|^2ds\Big),
\end{align}
where
\begin{align}
	I_1&=\int_{0}^{1}\big((K^H)^{-1}b(\tilde{X}_u,\phi^{(2)}_u+B_u^H)\big)(s)dW_s,\nonumber\\
	I_2&=\int_{0}^{1}-\dot{\phi}^{(2)}_sdW_s,\nonumber\\
	I_3&:=\int_{0}^{1}\dot{\phi}^{(2)}_s\cdot\Big((K^H)^{-1}\big(b(\tilde{X}_u,\phi^{(2)}_u+B_u^H)-b(\phi_u)\big)(s)\Big)ds,\nonumber\\
	I_4&:=\frac{1}{2}\int_{0}^{1}\Bigg(\Big(\big((K^H)^{-1}b(\phi_u)\big)(s)\Big)^2-\Big(\big((K^H)^{-1}b(\tilde{X}_u,\phi^{(2)}_u+B_u^H)\big)(s)\Big)^2\Bigg)ds,\nonumber
\end{align}
Before we continue our work, we first give a series of auxiliary lemmas required in different cases in Section \ref{key lemmas ans pro}.
\subsection{Key lemmas and propositions}\label{key lemmas ans pro}
\begin{lem}[\cite{IW14} \text{pp $536$-$537$}]\label{Separation lemma}
	For a fixed $n\geqslant1$, let $I_1, \ldots, I_n $ be $n$ random variables defined on $(\Omega,\mathcal{F},\mathbb{P})$ and $\{A_\varepsilon;\varepsilon >0\}$ a family of sets in $\mathcal{F}$. Suppose that for any $c\in \mathbb{R}$ and any $i=1,\dots, n$, if we have
	\begin{align}
		\limsup _{\varepsilon \rightarrow 0} E\Big(\exp(cI_i)\vert A_\varepsilon\Big)\leqslant 1.\nonumber
	\end{align}
	Then
	\begin{align}
		\lim_{\varepsilon\to 0} E \left [\exp \left(\sum_{i=1}^{n}cI_i \right ) \bigg| A_\varepsilon \right ]= 1.\nonumber
	\end{align}
\end{lem}
Before we introduce the following Theorem, we first recall some definitions and results on approximate limits in the Wiener space with respect to measurable semi-norms for exponentials of random variables in the first and second Wiener chaos.

Let $W=\{W_t,t\in[0,1]\}$ be a Wiener process defined in the canonical probability space $(\Omega,\mathcal{F},P)$. That is, $\Omega$ is the space of continuous functions vanishing
at zero and $P$ is the Wiener space. Let $H^1$ be the Cameron-Martin space, that is, the space of all absolutely continuous functions $h:[0,1]\to\mathbb{R}$ such that $h'\in H=L^2([0,1],\mathbb{R})$. The scalar product in $H^1$ is defined by
$$\langle h,g \rangle_{H^1}=\langle h',g' \rangle_{H},$$
for all $h,g\in H^1.$

Let $Q:H^1\to H^1$ be an orthogonal projection such that $\dim QH^1<\infty$. $Q$ can be written as
$$Qh=\sum_{i=1}^{n}\langle h, h_i\rangle h_i,$$
where $(h_1,\cdots, h_n)$ is an orthonormal sequence in $QH^1$. We can also define the $H^1$-valued random variable
$$Q^W=\sum_{i=1}^{n}\Big(\int_{0}^{1}h_i'(s)dW_s\Big)h_i.$$
Note that $Q^W$ does not depend on $(h_1,\cdots, h_n).$

A sequence of orthogonal projections $Q_n$ on $H^1$ is called the approximating sequence of projections if $\dim Q_nH^1<\infty$ and $Q_n$ converges strongly increasing to the identity operator in $H^1$.
\begin{defn}[\cite{MN02} Definition $1$]
	A semi-norm $N$ on $H^1$ is called a measurable semi-norm if there exists a random variable $\tilde{N}<\infty$ a.s, such that for all approximating sequence of projections $Q^n$ on $H^1$, the sequence $N(Q^W_n)$ converges in probability to $\tilde{N}$ and $P(\tilde{N}\leq\varepsilon)> 0$ for all $\varepsilon>0$. If moreover $N$ is a norm on $H^1$, then is called measurable norm.
\end{defn}
We will make use of the following result on measurable semi-norms.
\begin{lem}[\cite{MN02} Lemma $1$]
	Let $N_n$ be a nondecreasing sequence of measurable semi-norms. Suppose that $\tilde{N}:= P-\lim_{n\to\infty}\tilde{N}_n$ exists and $P (\tilde{N}\geq \varepsilon) > 0$ for all $\varepsilon>0$. Then $N = \lim_{n\to\infty}N_n$ is a measurable semi-norm if this limit exists on $H^1$.
\end{lem}
\begin{thm}[\cite{Har02}\label{no random function} Theorem $6$]
	Let $N$ be a measurable norm on $H^1$. Then,
	$$\lim_{\varepsilon\to 0}E\Big(\exp\Big(\int_{0}^{1}h(s)dW_s\Big)|\tilde{N}<\varepsilon\Big)=1,$$
	for all $h\in L^2([0,1]).$
\end{thm}
Moreover, we also need a stronger version of Theorem \ref{no random function}. That is,
\begin{thm}[\cite{Har04}\label{no random function 2} Example $3.9$]
	Let $N$ be a measurable norm on $H^1$. Then,
	$$\lim_{\varepsilon\to 0}E\Big(\exp\Big(\big|\int_{0}^{1}h(s)dW_s\big|\Big)\big|\tilde{N}<\varepsilon\Big)=1,$$
	for all $h\in L^2([0,1]).$
\end{thm}
We recall that an operator $K:H\to H$ is nuclear iff
$$\sum_{n=1}^{\infty}|\langle Ke_n,g_n\rangle|<\infty,$$
for all $B=(e_n)_n, B'=(g_n)_n$ orthonormal sequences in $H$. We define the trace of a nuclear operator $K$ by
$$Tr K=\sum_{i=1}^{\infty}\langle Ke_n,e_n\rangle,$$
for any $B=(e_n)_n$ orthonormal sequence in $H$. The definition is independent of the sequence we have chosen. Given a symmetric function $f\in L^2([0, 1]^2),$ the Hilbert-Schmidt operator $K(f): H \to H$ associated with $f$, defined by
\begin{align}
	(K(f))(h)(t)=\int_{0}^{t}f(t,u)h(u)du,\nonumber
\end{align}
is nuclear iff $\sum_{n=1}^{\infty}|\langle Ke_n,g_n\rangle|<\infty$ for all $B=(e_n)_n$ orthonormal sequence in $H$. If $f$ is continuous and the operator $K(f)$ is nuclear, we can compute its trace as follows:
$$Tr(f):=Tr K(f)=\int_{0}^{1}f(s,s)ds.$$
\begin{lem}[\cite{Har02}\label{key lemma} Theorem $8$]
	Let $f$ be a symmetric function in $L^2([0,1]^2)$ and let $N$ be a measurable norm. If $K(f)$ is nuclear, then
	\begin{align}
		\lim_{\varepsilon\to 0} E \Big(\exp\Big(\int_{0}^{1}\int_{0}^{1}f(s,t)dW_s dW_t\Big)|\tilde{N}<\varepsilon\Big)= e^{-Tr(f)}.\nonumber
	\end{align}
\end{lem}

\begin{lem}[\cite{LL98} \text{Theorem $1.1$}]
	Let $B^H$ be a fractional Brownian motion. Then
	$$\lim_{\varepsilon\to 0}\varepsilon^{\frac{1}{H}}\log P\Big(\sup\limits_{t \in[0, 1]}\big|B^H\big| \leq \varepsilon\Big)=-C_H,$$
	where $C_H$ is a positive constant.
\end{lem}
\begin{lem}[\cite{KLS95} \text{Lemma $3.1$}]\label{small ball of holder}
	Let $B^H$ be a fractional Brownian motion and let $0\leq \beta<H$. Then there exists constants $0<K_1\leq K_2<\infty$ depending on $H$ and $\beta$ such that for all $0<\varepsilon<1$
	$$-K_2\varepsilon^{-\frac{1}{H-\beta}}\leq \log P\Big(\sup\limits_{0\leq s,t\leq 1}\big|\frac{|B^H_t-B^H_s|}{|t-s|^{\beta}}\big| \leq \varepsilon\Big)\leq -K_1\varepsilon^{-\frac{1}{H-\beta}}.$$
\end{lem}
\begin{lem}\label{equivalence}
	Assume $\|Z-\phi\|_2\leq \varepsilon$ defined as above, then we have the following equality associated small probabilities.
	\begin{align}
		P(\|Z-\phi\|\leq\varepsilon)= P(\|Y-\phi^{(2)}\|\leq\varepsilon),\nonumber
	\end{align}
	where $\|\cdot\|$ denote $\|\cdot\|_{\infty}$ or $\|\cdot\|_{\beta}$.
\end{lem}
\begin{proof}
	Firstly, it is clear that if $\|Z-\phi\|_2\leq\varepsilon$, then
	$\|Y-\phi^{(2)}\|\leq\varepsilon$. So in order to prove Lemma \ref{equivalence}, we only need to prove that if $\|Y-\phi^{(2)}\|\leq\varepsilon$, then $\|X-\phi^{(1)}\|\leq\varepsilon$ under given two different norms $\|\cdot\|_{\infty}$ and $\|\cdot\|_{\beta}$.
	
	$\bullet$ Supremum norm $\|\cdot\|_{\infty}$.
	
	Recalling the structure of reference path and using that $\sigma$ is Lipschitz with constant $L_1$, by H$\mathrm{\ddot{o}}$lder inequality we have
	\begin{align}
		|X_t-\phi^{(1)}_t|^2=\Big|\int_{0}^{t}(\sigma(X_s,Y_s)-\sigma(\phi_s))ds\Big|^2\leq L_1^2\int_{0}^{t}\Big(|X_s-\phi_s^{(1)}|^2+|Y_s-\phi_s^{(2)}|^2\Big)ds.\nonumber
	\end{align}
	By Gronwall's inequality, we obtain
	\begin{align}
		|X_t-\phi^{(1)}_t|^2\leq e^{L_1^2 t}\int_{0}^{t}|Y_s-\phi_s^{(2)}|^2ds\leq C(L_1)\|Y-\phi^{(2)}\|^2_{\infty}.\nonumber
	\end{align}
	So we easily get if $\|Y-\phi^{(2)}\|_{\infty}\leq\varepsilon$, then we have $$\|X-\phi^{(1)}\|_{\infty}\leq\varepsilon.$$
	
	$\bullet$ H$\mathrm{\ddot{o}}$lder norm $\|\cdot\|_{\beta}$.
	
	Since the supremum norm has been verified above, we only need to verify H$\mathrm{\ddot{o}}$lder seminorm.\\
	It is easy to see that
	\begin{align}
		|X_t-\phi^{(1)}_t-(X_s-\phi^{(1)}_s)|=\Big|\int_{s}^{t}(\sigma(X_s,Y_s)-\sigma(\phi_s))ds\Big|\leq L_1\int_{0}^{t}\Big(|X_s-\phi_s^{(1)}|+|Y_s-\phi_s^{(2)}|\Big)ds.\nonumber
	\end{align}
	By Gronwall's inequality, we have
	\begin{align}
		|X_t-\phi^{(1)}_t-(X_s-\phi^{(1)}_s)|\leq e^{L_1(t-s)}\int_{s}^{t}|Y_u-\phi^{(2)}_u|d u.\nonumber
	\end{align}
	By the definition of H$\mathrm{\ddot{o}}$lder seminorm, we obtain
	\begin{align}
		\|X-\phi^{(1)}\|_{\beta}\leq \frac{e^{L_1(t-s)}}{|t-s|^{\beta}} \int_{s}^{t}|Y_u-\phi^{(2)}_u|d u.\nonumber
	\end{align}
	Let $v=\frac{u-s}{t-s}$, we have
	\begin{align}
		\|X-\phi^{(1)}\|_{\beta}&\leq\frac{e^{L_1(t-s)}}{|t-s|^{\beta}}\int_{0}^{1}|Y_{(t-s)v+s}-\phi^{(2)}_{(t-s)v+s}||t-s|d v
		\nonumber\\&\leq e^{L_1(t-s)} |t-s|^{1-\beta}\int_{0}^{1}|Y_{(t-s)v+s}-\phi^{(2)}_{(t-s)v+s}|d v\nonumber\\&\leq C(L_1) \|Y-\phi^{(2)}\|_{\beta}\nonumber.
	\end{align}
	So we obtain if $\|Y-\phi^{(2)}\|_{\beta}\leq\varepsilon$, then we have $$\|X-\phi^{(1)}\|_{\beta}\leq\varepsilon.$$
\end{proof}

Then we define the following norms on $H^1$:
$$N_H(h)=\sup\limits_{t \in[0, 1]}\big|\int_{0}^{t}K^H(t,s)h'(s)ds\big|,$$
$$N_{H,\beta}(h)=\sup\limits_{t,r\in [0, 1]}\frac{\Big|\int_{0}^{t}K^H(t,s)h'(s)ds-\int_{0}^{r}K^H(t,s)h'(s)ds\Big|}{|t-r|^\beta},$$
for $0<\beta<H.$
\begin{lem}[\cite{MN02} \text{Lemma 6}]\label{norm}
	$N_H$ and $N_{H,\beta}$ with $0<\beta\leq H$ are measurable norms and we have $\tilde{N}_H=\|B^H\|_{\infty}$ and $\tilde{N}_{H,\beta}=\|B^H\|_{\beta}$.
\end{lem}
\begin{lem}\label{girsanov}
	Let $\eta$ be the process defined by \eqref{eta}. Then $\eta$ is adapted and
	$$E\Big(\exp\Big(\int_{0}^{1}\eta_s dW_s-\frac{1}{2}\int_{0}^{1}\eta_s^2ds\Big)\Big)=1.$$
\end{lem}
\begin{proof}
	The proof of this lemma is similar to Lemma $10$ of \cite{MN02}. So the proof will be omitted.
\end{proof}
\begin{lem}[\cite{MN02} \text{Lemma 11}]\label{space of singular case}
	Assume $H<\frac{1}{2}$. The space $\mathcal{H}=\{K^Hh,h\in L^p([0,1])\}$ defined for $p>\frac{1}{H+\frac{1}{2}}$ is included in the space of H$\mathrm{\ddot{o}}$lder continuous functions of order $H+\frac{1}{2}-\dfrac{1}{p}$.
\end{lem}

\begin{lem}\label{trace class of singular case}
	Assume $\frac{1}{4}<H<\frac{1}{2}$. Let $f$ be the function defined by
	$$	f(s,r)=\frac{1}{\G(\a)}s^{-\a}\int_{r}^{s}u^{\a}b_y(\phi_u)(s-u)^{\a-1}K^{H}(u,r)du \ \mathbb{I}_{s\geq r},$$
	where $s,r\in (0,1]$ and $\phi$ is such that $\phi^{(2)}-y\in\mathcal{H}^p$ for $p>\frac{1}{H}$. Let $\tilde{f}$ be the symmetrization of $f$. Then the operator $K(\tilde{f})$ defined by $K(\tilde{f})(h)(s)=\int_{0}^{s}\tilde{f}(s,r)h(r)dr$ is nuclear.
\end{lem}
\begin{proof}
	The proof of this lemma is similar to Lemma $13$ of \cite{MN02}. So the proof will be omitted.
\end{proof}
\begin{lem}\label{trace class of regular case}
	Assume $H\geq\frac{1}{2}$. Let $f$ be the function defined by
	\begin{align}
		f(s,r)=&\frac{1}{\G(1-\a)}\Big(s^{-\a}b_y(\phi_s)K^H(s,r)\mathbb{I}_{r\leq s}+\a s^{\a}\nonumber\\&\cdot \int_{0}^{s}\frac{s^{-\a}b_y(\phi_s)K^H(s,r)\mathbb{I}_{r\leq s}-u^{-\a}b_y(\phi_u)K^H(u,r)}{(s-u)^{\a+1}}du\Big)\nonumber
	\end{align}
	where $s,r\in (0,1]$ and $\phi$ is such that $\phi^{(2)}-y\in\mathcal{H}^2$. Let $\tilde{f}$ be the symmetrization of $f$. Then the operator $K(\tilde{f})$ defined by $K(\tilde{f})(h)(s)=\int_{0}^{s}\tilde{f}(s,r)h(r)dr$ is nuclear.
\end{lem}
\begin{proof}
	The proof of this lemma is similar to Lemma $14$ of \cite{MN02}. So the proof will be omitted.
\end{proof}
\section{Proof of Theorem \ref{th:singular case} and \ref{th:regular case}}
In this section, we will derive Onsager-Machlup action functional for DSDE $(\ref{dsde})$ under singular case ($H<\frac{1}{2}$) and regular case ($H>\frac{1}{2}$).

\subsection{Proof of Theorem \ref{th:singular case}}
We will prove the theorem for H$\mathrm{\ddot{o}}$lder norm. The proof is the same for the supremum norm. Recall operator $(K^H)^{-1}$ is defined by
$$\Big(\big(K^H\big)^{-1}h\Big)(s)=s^{-\alpha}\big(I^{\alpha}_{0+}u^{\alpha}h'\big)(s).$$
where $\a=\frac{1}{2}-H$ when $H<\frac{1}{2}.$
For simplicity of presentation, we define the error between $\tilde{X}$ and $\phi^{(1)}$ as:
$$\kappa_s=\tilde{X}_s-\phi^{(1)}_s, 0\leq s\leq 1.$$
\\
Then similar to the proof of Lemma \ref{equivalence}, we have
\begin{align}
	\|\kappa\|_{\infty} \leq C(L_1)\|B^H\|_{\infty},\nonumber	
\end{align}
and
\begin{align}\label{holder of kappa}
	\|\kappa\|_{\b} \leq C(L_1)\|B^H\|_{\b}.
\end{align}
So we can rewrite small probability \eqref{all}.
\begin{align}\label{small all}
	&P(\|Y-\phi^{(2)}\|_{\beta}\leq\varepsilon)\nonumber\\&=E\Big(\exp(I_1+I_2+I_3+I_4)\mathbb{I}_{\|B^H\|\leq\varepsilon}\Big)\times\exp\Big(-\frac{1}{2}\int_{0}^{1}\big|\dot{\phi}^{(2)}_s-s^{-\a}\big(I^{\a}_{0+}u^{\a}b(\phi_u)\big)(s)\big|^2ds\Big),
\end{align}
where
\begin{align}
	I_1&=\int_{0}^{1}s^{-\alpha}\big(I^{\alpha}_{0+}u^{\alpha}b(\phi^{(1)}_u+\kappa_u,\phi^{(2)}_u+B_u^H)\big)(s)dW_s,\nonumber\\
	I_2&=\int_{0}^{1}-\dot{\phi}^{(2)}_sdW_s,\nonumber\\
	I_3&=\int_{0}^{1}\dot{\phi}^{(2)}_s\cdot\Big(s^{-\alpha}\big(I^{\alpha}_{0+}u^{\alpha}\big(b(\phi^{(1)}_u+\kappa_u,\phi^{(2)}_u+B_u^H)-b(\phi_u)\big)\big)(s)\Big)ds,\nonumber\\I_4&=\frac{1}{2}\int_{0}^{1}\Bigg(\Big(s^{-\alpha}\big(I^{\alpha}_{0+}u^{\alpha}b(\phi_u)\big)(s)\Big)^2-\Big(s^{-\a}\big(I^{\alpha}_{0+}u^{\alpha}b(\phi^{(1)}_u+\kappa_u,\phi^{(2)}_u+B_u^H)\big)(s)\Big)^2\Bigg)ds,\nonumber\\\nonumber
\end{align}
Then by applying Lemma \ref{Separation lemma}, we could deal with each term independently.

$\clubsuit \ \text{Term}$ \ $I_2$

Applying Theorem \ref{no random function} to $f=-c\dot{\phi}^{(2)}_s$ and Lemma \ref{norm}, we have
\begin{align}\label{I_2}
	\limsup _{\varepsilon \rightarrow 0}\ E(\exp(cI_2)|\|B^H\|_{\beta}<\varepsilon)\leq1,
\end{align}
for every real number $c$.

$\clubsuit  \ \text{Term}$ \ $I_3$

So under the condition $\|B^H\|_{\beta}<\varepsilon$, by using the fact $\sigma,b$ are Lipschitz continuous and bounded with constant $L_1,L_2$ respectively, we have that (where $v=\frac{u}{s}$)
\begin{align}\label{ready1}
	&\Big|s^{-\alpha}\big(I^{\alpha}_{0+}u^{\alpha}\big(b(\phi^{(1)}_u+\kappa_u,\phi^{(2)}_u+B_u^H)-b(\phi_u)\big)\big)(s)\Big|\nonumber\\=&\frac{1}{\G(\a)}s^{-\alpha}\Bigg|\int_{0}^{s}u^{\alpha}(s-u)^{\a-1}\Big(b(\phi^{(1)}_u+\kappa_u,B_u^H+\phi_u^{(2)})-b(\phi_u)\Big)du\Bigg|\nonumber\\\leq&\frac{C(L_1,L_2)}{\G(\a)} s^{-\alpha}\int_{0}^{s}u^{\alpha}(s-u)^{\a-1}\big(\|\kappa\|_{\b}+\|B^H\|_{\b}\big)du\nonumber\\=&\frac{2C(L_1,L_2)}{\G(\a)}\varepsilon s^{\alpha}\int_{0}^{1}v^{\alpha}(1-v)^{\a-1}dv=2C(L_1,L_2)\frac{\b(1+\a,\a)}{\G(\a)} s^{\alpha}\varepsilon.
\end{align}
We now deal with the term $I_3$.
\begin{align}
	|I_3|&=\Bigg|\int_{0}^{1}\dot{\phi}^2_s\Big(s^{-\alpha}\big(I^{\alpha}_{0+}u^{\alpha}\big(b(\phi^{(1)}_u+\kappa_u,\phi^{(2)}_u+B_u^H)-b(\phi_u)\big)\big)(s)\Big)ds\Bigg|\nonumber\\&\leq 2C(L_1,L_2)\frac{\b(1+\a,\a)}{\G(\a)} s^{\alpha}\varepsilon \int_{0}^{1}s^{\a}|\dot{\phi}^{(2)}_s|ds\nonumber\\&\leq C(L_1,L_2,\alpha)\varepsilon.\nonumber
\end{align}
Hence
\begin{align}\label{I_3}
	\limsup _{\varepsilon \rightarrow 0}\ E(\exp(cI_3)|\|B^H\|<\varepsilon)\leq1,
\end{align}
for every real number $c$.

$\clubsuit \ \text{Term} $\ $I_4$

For the term $I_4$, we have
\begin{align}
	|I_4|\leq& \frac{1}{2}\int_{0}^{1}\Bigg|\Big(s^{-\alpha}\big(I^{\alpha}_{0+}u^{\alpha}b(\phi_u)\big)(s)\Big)^2-\Big(s^{-\a}\big(I^{\alpha}_{0+}u^{\alpha}b(\phi^{(1)}_u+\kappa_u,\phi^{(2)}_u+B_u^H)\big)(s)\Big)^2\Bigg|ds\nonumber\\\leq& \frac{1}{2}\int_{0}^{1}\Bigg(s^{-\alpha}\Big(I^{\alpha}_{0+}u^{\alpha}\big(b(\phi^{(1)}_u+\kappa_u,\phi^{(2)}_u+B_u^H)-b(\phi_u)\big)\Big)(s)\Bigg)^2ds\nonumber\\&+\int_{0}^{1}\Bigg|s^{-\a}\Big(I^{\a}_{0+}u^{\a}\big(b(\phi^{(1)}_u+\kappa_u,B_u^H+\phi^{(2)}_u)-b(\phi_u)\big)\Big)(s)\cdot s^{-\a}\big(I^{\a}_{0+}u^{\a}b(\phi_u)\big)(s)\Bigg|ds\nonumber\\\nonumber:=& I_{41}+I_{42}.
\end{align}
Using \eqref{ready1} we obtain
\begin{align}
	|I_{41}|=&\frac{1}{2}\int_{0}^{1}\Bigg(s^{-\alpha}\Big(I^{\alpha}_{0+}u^{\alpha}\big(b(\phi^{(1)}_u+\kappa_u,\phi^{(2)}_u+B_u^H)-b(\phi_u)\big)\Big)(s)\Bigg)^2ds\nonumber\\\leq& C^2(L_1,L_2)\frac{\b^2(1+\a,\a)}{(2\a+1)\G(\a)^2}\varepsilon^2,\nonumber
\end{align}
and from \eqref{fractional integral estimate}, \eqref{ready1} and $b$ is bounded, we have ($p=1,f=s^{\a}b(\phi_u)$)
\begin{align}
	|I_{42}|=&\int_{0}^{1}\Bigg|s^{-\a}\Big(I^{\a}_{0+}u^{\a}\big(b(\phi^{(1)}_u+\kappa_u,B_u^H+\phi^{(2)}_u)-b(\phi_u)\big)\Big)(s)s^{-\a}\big(I^{\a}_{0+}u^{\a}b(\phi_u)\big)(s)\Bigg|ds\nonumber\\\leq&2C(L_1,L_2)\frac{\beta(1+\a,\a)}{\G(\a)}\varepsilon\int_{0}^{1}\big(I^{\a}_{0+}u^{\a}b(\phi_u)\big)(s)ds\nonumber\\\leq&2C(L_1,L_2)\frac{\beta(1+\a,\a)}{\a\G(\a)^2}\varepsilon\int_{0}^{1}s^{\a}b(\phi_s)ds\leq C(L_1,L_2,\a)\varepsilon. \nonumber
\end{align}
As a consequence, by Lemma \ref{Separation lemma} we get that
\begin{align}\label{I_4}
	\limsup _{\varepsilon \rightarrow 0}\ E(\exp(cI_4)|\|B^H\|<\varepsilon)\leq1,
\end{align}
for every real number $c$.

$\clubsuit \ \text{Term}$ \ $I_1$

Applying classical Taylor expansion to $b(\phi^{(1)}_s+\kappa_s,B_s^H+\phi^{(2)}_s)$ at $\phi$ we have
\begin{align}
	b(\phi^{(1)}_s+\kappa_s,B_s^H+\phi^{(2)}_s)=b(\phi_s)+b_x(\phi_s)\kappa_s+b_y(\phi_s)B_s^H+R_s^{x,y},\nonumber
\end{align}
where $R^{x,y}$ denotes the remainder term. If $\|B^H\|_{\b}\leq \varepsilon$, by Young's inequality we have that
\begin{align}\label{the bound of remainder term}
	\|R\|_{\infty}\leq C(L_1)\varepsilon^2.
\end{align}
We now rewrite the term $I_1$.
\begin{align}\label{combo of I_1}
	I_1&=\int_{0}^{1}s^{-\alpha}\big(I^{\alpha}_{0+}u^{\alpha}b(\phi^{(1)}_u+\kappa_u,\phi^{(2)}_u+B_u^H)\big)(s)dW_s\nonumber\\&=\int_{0}^{1}s^{-\alpha}\Big(I^{\alpha}_{0+}u^{\alpha}\big(b(\phi_u)+b_x(\phi_u)\kappa_u+b_y(\phi_u)B_u^H+R_u^{x,y}\big)\Big)(s)dW_s\nonumber\\&:=I_{11}+I_{12}+I_{13}+I_{14},
\end{align}
Applying Theorem \ref{no random function} to $f=cs^{-\a}\big(I^{\a}_{0+}u^{\a}b(\phi_u)\big)(s)$ and Lemma \ref{Separation lemma}, we have
\begin{align}\label{I_{11}}
	\limsup _{\varepsilon \rightarrow 0}\ E(\exp(cI_{11})|\|B^H\|_{\beta}<\varepsilon)\leq1,
\end{align}
for every real number $c$.\\
We now deal with the term $I_{12}$. By using the fact $b_x$ is bounded and $\|\kappa\|_{\b} \leq C(L_1)\|B^H\|_{\b}$ so we have
\begin{align}
	I_{12}=&\int_{0}^{1}s^{-\alpha}\Big(I^{\alpha}_{0+}u^{\alpha}\big(b_x(\phi_u)\kappa_u\big)\Big)(s)dW_s\nonumber\\=&\frac{1}{\G(\a)}\int_{0}^{1}s^{-\a}\int_{0}^{s}u^{\a}b_x(\phi_u)\kappa_u(s-u)^{\a-1}dudW_s\nonumber\\=&\frac{1}{\G(\a)}\int_{0}^{1}u^{\a}b_x(\phi_u)\kappa_u\int_{u}^{1}s^{-\a}(s-u)^{\a-1}dW_sdu,&\nonumber\\\leq&\max\Big\{\frac{1}{\G(\a)}\int_{0}^{1}u^{\a}C\varepsilon\int_{u}^{1}s^{-\a}(s-u)^{\a-1}dW_sdu, -\frac{1}{\G(\a)}\int_{0}^{1}u^{\a}C\varepsilon\int_{u}^{1}s^{-\a}(s-u)^{\a-1}dW_sdu\Big\},\nonumber
\end{align}
Then we further obtain
\begin{align}
	0\leq |I_{12}|&\leq \Big|\frac{1}{\G(\a)}\int_{0}^{1}u^{\a}C\varepsilon\int_{u}^{1}s^{-\a}(s-u)^{\a-1}dW_sdu\Big|\nonumber\\&\leq\Big|\frac{1}{\G(\a)}\int_{0}^{1}C\varepsilon s^{-\a}\int_{0}^{s}u^{\a} (s-u)^{\a-1}dudW_s\Big|\nonumber\\&=\Big|\int_{0}^{1}C\varepsilon s^{\a}\frac{\b(1+\a,\a)}{\G(\a)}dW_s\Big|.\nonumber
\end{align}
Applying Theorem \ref{no random function 2} to $h=C\varepsilon s^{\a}\frac{\b(1+\a,\a)}{\G(\a)}$, we have
\begin{align}\label{I_{12}}
	\limsup _{\varepsilon \rightarrow 0}\ E(\exp(cI_{12})|\|B^H\|<\varepsilon)\leq1,
\end{align}
for every real number $c$.
\\
In order to study the limit behavior of the conditional exponential moments of the term $I_{13}$. We will express $I_{13}$ as a double stochastic integral with respect to $W$ based on the integral representation of fractional Brownian motion $B^H$, so we obtain
\begin{align}
	I_{13}=&\int_{0}^{1}s^{-\alpha}\Big(I^{\alpha}_{0+}u^{\alpha}\big(b_y(\phi_u)B^H_u\big)\Big)(s)dW_s\nonumber\\=&\frac{1}{\G(\a)}\int_{0}^{1}s^{-\a}\int_{0}^{s}u^{\a}b_y(\phi_u)B^H_u(s-u)^{\a-1}dudW_s\nonumber\\=&\frac{1}{\G(\a)}\int_{0}^{1}s^{-\a}\int_{0}^{s}u^{\a}b_y(\phi_u)(s-u)^{\a-1}\int_{0}^{u}K^{H}(u,r)dW_rdudW_s\nonumber\\=&\frac{1}{\G(\a)}\int_{0}^{1}s^{-\a}\int_{0}^{s}\int_{r}^{s}u^{\a}b_y(\phi_u)(s-u)^{\a-1}K^{H}(u,r)dudW_rdW_s\nonumber\\=&\int_{0}^{1}\int_{0}^{1}f(s,r)dW_rdW_s=\int_{0}^{1}\int_{0}^{1}\tilde{f}(s,r)dW_rdW_s,\nonumber
\end{align}
where $\tilde{f}$ is the symmetrization of the function ($\tilde f(s,r)=\tilde f(r,s)$)
$$
f(s,r)=\frac{1}{\G(\a)}s^{-\a}\int_{r}^{s}u^{\a}b_y(\phi_u)(s-u)^{\a-1}K^{H}(u,r)du \ \mathbb{I}_{s\geq r}.
$$
By Lemma \ref{trace class of singular case} the operator $K(\tilde{f})$ is nuclear. So the trace of this operator can be obtained as
$$Tr\tilde{f}=\int_{0}^{1}\tilde{f}(s,s)ds=\frac{1}{2}\int_{0}^{1}f(s,s)ds.$$
Note that the function $\tilde{f}$ is not continuous on the axes, but the result of \cite{Bal76} still holds in this case taking into account the particular form of the function $f$. In order to compute the integral $\int_{0}^{1}f(s,s)ds.$ Let us rewrite $f(s,r)$ by the expression \eqref{K^H(r,u)} of the kernel $K^H$,
\begin{align}
	f(s,r)=&\frac{1}{\G(\a)}s^{-\a}\int_{r}^{s}u^{\a}b_y(\phi_u)(s-u)^{\a-1}K^{H}(u,r)du \ \mathbb{I}_{s\geq r}\nonumber\\=&\frac{c_H}{\G(\a)}s^{-\a}\int_{r}^{s}u^{\a}b_y(\phi_u)(s-u)^{\a-1}(u-r)^{-\a}du \ \mathbb{I}_{s\geq r}\nonumber\\&+\frac{c_H\a}{\G(\a)}s^{-\a}\int_{r}^{s}\int_{r}^{u}u^{\a}b_y(\phi_u)(s-u)^{\a-1}(\theta-r)^{-\a-1}\big(1-(\frac{r}{\theta})^{\a}\big)d\theta du \ \mathbb{I}_{s\geq r}\nonumber\\:=&(f_1+f_2)(s,r).\nonumber
\end{align}
The change of variable $w=\frac{u-r}{s-r}$ yields
\begin{align}
	f_1(s,r)=&\frac{c_H}{\G(\a)}s^{-\a}\int_{r}^{s}u^{\a}b_y(\phi_u)(s-u)^{\a-1}(u-r)^{-\a}du \ \mathbb{I}_{s\geq r}\nonumber\\=&\frac{c_H}{\G(\a)}s^{-\a}\int_{0}^{1}((s-r)w+r)^{\a}b_y(\phi_{(s-r)w+r)}(1-w)^{\a-1}w^{-\a}dw \ \mathbb{I}_{s\geq r},\nonumber
\end{align}
and hence, we have
\begin{align}
	f_1 (s,s)=\frac{c_H}{\G(\a)}b_y(\phi_{s})\int_{0}^{1}(1-w)^{\a-1}w^{-\a}dw=c_H\G(1-\a)b_y(\phi_{s}).\nonumber
\end{align}
On the other hand,  the change of variable $v=\frac{\theta-r}{u-r}$ yields
\begin{align}
	f_2(s,r)=&\frac{c_H\a}{\G(\a)}s^{-\a}\int_{r}^{s}\int_{r}^{u}u^{\a}b_y(\phi_u)(s-u)^{\a-1}(\theta-r)^{-\a-1}\big(1-(\frac{r}{\theta})^{\a}\big)d\theta du \ \mathbb{I}_{s\geq r}\nonumber\\=&\frac{c_H\a}{\G(\a)}s^{-\a}\int_{r}^{s}u^{\a}b_y(\phi_u)(s-u)^{\a-1}(u-r)^{-\a}B(r,u)du \ \mathbb{I}_{s\geq r},\nonumber
\end{align}
where
$$B(r,u)=\int_{0}^{1}v^{-\a-1}\Big(1-\big(\frac{r}{(u-r)v+r}\big)\Big)^{\a}dv.$$
Introducing the change of variable $x=\frac{u-r}{s-r}$, we have
\begin{align}
	&f_2(s,r)=\frac{c_H\a}{\G(\a)}s^{-\a}\int_{0}^{1}\big((s-r)x+r\big)^{\a}b_y(\phi_{(s-r)x+r})(1-x)^{\a-1}x^{-\a}B(r,(s-r)x+r)dx \ \mathbb{I}_{s\geq r},\nonumber
\end{align}
so as variable $s=r$ we easily get
$$f_2(s,s)=\a c_H\G(1-\a)b_y(\phi_{s})B(s,s)=0.$$
As a consequence, we have
$$Tr(\tilde{f})=\frac{1}{2}\int_{0}^{1}f(s,s)ds=\frac{c_H\G(1-\a)}{2}\int_{0}^{1}b_y(\phi_{s})ds=\frac{d_H}{2}\int_{0}^{1}b_y(\phi_{s})ds.$$
To summarized what we have proved, Lemma \ref{key lemma} and Lemma \ref{norm} give us
\begin{align}\label{I_{13}}
	\limsup _{\varepsilon \rightarrow 0}\ E(\exp(I_{13})|\|B^H\|<\varepsilon)=\exp\Big(-\frac{d_H}{2}\int_{0}^{1}b_y(\phi_{s})ds\Big).
\end{align}
Finally, it only remains to study the limit behavior of the term $I_{14}$. For any $c\in \mathbb{R}$ and $\delta>0$ we can write
\begin{align}
	E&\Big(\exp(cI_{14})|\|B^H\|\leq \varepsilon\Big)\nonumber\\&\leq e^{\delta}+\int_{\delta}^{\infty}e^{\xi}P(|cI_{14}|>\xi|\|B^H\|_{\b}\leq \varepsilon)d \xi \nonumber\\&+e^{\delta}\mathbb{P}(|cI_{14}|>\delta |\|B^H\|_{\b}\leq \varepsilon).\nonumber
\end{align}
Define the martingale $M_t=c\int_{0}^{t}s^{-\a}\Big(I^{\alpha}_{0+}u^{\alpha}R_u^{x,y}\Big)(s)dW_s$ whose quadratic variations can be estimated by \eqref{the bound of remainder term} as follows,
\begin{align}
	\langle M \rangle_t=&c^2\int_{0}^{t}\Big(s^{-\alpha}\Big(I^{\alpha}_{0+}u^{\alpha}R_u^{x,y}\Big)(s)\Big)^2ds\nonumber\\ &\leq \frac{c^2 C(L_1)^2\b(\a,\a+1)^2}{(1+2\a)(\b(\a))^2}\varepsilon^4=C(L_1,\a)\varepsilon^4.\nonumber
\end{align}
Applying the exponential inequality for martingales, we have
\begin{align}\label{exponential inequality for martingales1}
	\mathbb{P}\Big(\Big|c\int_{0}^{t}s^{-\a}\Big(I^{\alpha}_{0+}u^{\alpha}R_u^{x,y}\Big)(s)dW_s\Big|>\xi, \|B^H\|_{\b}\leq \varepsilon\Big)\leq \exp\Big(-\frac{\xi^2}{2C(L_1,\a)\varepsilon^{4}}\Big),
\end{align}
for every real number $c$.\\
Combining Lemma \ref{small ball of holder} and inequality \eqref{exponential inequality for martingales1}, we see that
\begin{align}
	\mathbb{P}&\Big(\Big|c\int_{0}^{1}s^{-\a}\Big(I^{\alpha}_{0+}u^{\alpha}R_u^{x,y}\Big)(s)dW_s\Big|>\xi,\Big|\|B^H\|_{\b}\leq \varepsilon\Big)\nonumber\\&\leq \exp\Big(-\frac{\xi^2}{2C(L_1,\a)\varepsilon^{4}}\Big)\exp\big(	C_H \varepsilon^{-\frac{1}{H-\b}}\big).\nonumber
\end{align}
\\
Using the latter estimate we have for every $\delta>0$ and every $0<\varepsilon<1$
\begin{align}
	E&\Big(\exp(cI_{14})|\|B^H\|_{\beta}\leq \varepsilon\Big)\nonumber\\&\leq e^{\delta}+\int_{\delta}^{\infty}\exp\Big\{\xi-\frac{\xi^2}{2C(L_1,\a)\varepsilon^{4}}+C_H \varepsilon^{-\frac{1}{H-\b}}\Big\}d \xi \nonumber\\&+\exp\Big\{\delta-\frac{\delta^2}{2C(L_1,\a)\varepsilon^{4}}+C_H \varepsilon^{-\frac{1}{H-\b}}\Big\}.\nonumber
\end{align}
Letting $\varepsilon$ and then $\delta$ tend to zero, we obtain
\begin{align}\label{I_{14}}
	\limsup _{\varepsilon \rightarrow 0}\ E(\exp(cI_{14})|\|B^H\|_{\b}<\varepsilon)\leq1,
\end{align}
for every real number $c$. \\
Finally, we can summarize what we have derived, \eqref{small all},\eqref{I_2}, \eqref{I_3}, \eqref{I_4}, \eqref{combo of I_1},\eqref{I_{11}},\eqref{I_{12}},\eqref{I_{13}} and \eqref{I_{14}} give us
$$
\lim _{\varepsilon \rightarrow 0} \frac{\mathbb{P}(\|Y-\phi^{(2)}\|_{\b}<\varepsilon)}{\mathbb{P}(\|B^H\|_{\b}<\varepsilon)}=\exp \left(-\frac{1}{2}\int_{0}^{1}\big|\dot{\phi}^{(2)}_s-s^{-\a}\big(I^{\a}_{0+}u^{\a}b(\phi_u)\big)(s)\big|^2ds-\frac{d_H}{2}\int_{0}^{1}b_y(\phi_{s})ds\right).
$$
The proof of Theorem \ref{th:singular case} is complete.
\qed
\subsection{Proof of Theorem \ref{th:regular case}}
The proof of this theorem can be completed by the method analogous to that used above. Recall operator $(K^H)^{-1}$ is defined by
$$\Big(\big(K^H\big)^{-1}h\Big)(s)=s^{\alpha}\big(D^{\alpha}_{0+}u^{-\alpha}h'\big)(s).$$
where $\a=H-\frac{1}{2}$ when $H>\frac{1}{2}.$ So we can rewrite small probability \eqref{all} as
\begin{align}\label{small all 2}
	&P(\|Y-\phi^{(2)}\|_{\beta}\leq\varepsilon)\nonumber\\&=E\Big(\exp(J_1+J_2+J_3+J_4)\mathbb{I}_{\|B^H\|\leq\varepsilon}\Big)\times\exp\Big(-\frac{1}{2}\int_{0}^{1}\big|\dot{\phi}^2_s-s^{\a}\big(D^{\a}_{0+}u^{-\a}b(\phi_u)\big)(s)\big|^2ds\Big),
\end{align}
where
\begin{align}
	J_1&=\int_{0}^{1}s^{\alpha}\big(D^{\alpha}_{0+}u^{-\alpha}b(\phi^{(1)}_u+\kappa_u,\phi^{(2)}_u+B_u^H)\big)(s)dW_s,\nonumber\\
	J_2&=\int_{0}^{1}-\dot{\phi}^{(2)}_sdW_s,\nonumber\\
	J_3&=\int_{0}^{1}\dot{\phi}^{(2)}_s\cdot\Big(s^{\alpha}\big(D^{\alpha}_{0+}u^{-\alpha}\big(b(\phi^{(1)}_u+\kappa_u,\phi^{(2)}_u+B_u^H)-b(\phi_u)\big)\big)(s)\Big)ds,\nonumber\\J_4&=\frac{1}{2}\int_{0}^{1}\Bigg(\Big(s^{\alpha}\big(D^{\alpha}_{0+}u^{\alpha}b(\phi_u)\big)(s)\Big)^2-\Big(s^{\a}\big(D^{\alpha}_{0+}u^{\alpha}b(\phi^{(1)}_u+\kappa_u,\phi^{(2)}_u+B_u^H)\big)(s)\Big)^2\Bigg)ds,\nonumber\\\nonumber
\end{align}
Then by applying Lemma \ref{Separation lemma}, we could deal with each term independently.

$\clubsuit \ \text{Term}$ \ $J_2$

Applying Theorem \ref{no random function} to $f=-c\dot{\phi}^{(2)}_s$ and Lemma \ref{norm}, we have
\begin{align}\label{J_2}
	\limsup _{\varepsilon \rightarrow 0}\ E(\exp(cJ_2)|\|B^H\|_{\beta}<\varepsilon)\leq1,
\end{align}
for every real number $c$.

$\clubsuit \ \text{Term}$ \ $J_3$

Using the Weyl representation \eqref{weyl representation} for the fractional derivative we have that
\begin{align}
	\Big|&s^{\alpha}\big(D^{\alpha}_{0+}u^{-\alpha}\big(b(\phi^{(1)}_u+\kappa_u,\phi^{(2)}_u+B_u^H)-b(\phi_u)\big)\big)(s)\Big|\nonumber\\&=\frac{1}{\G(1-\a)}\Bigg|\frac{b(\phi^{(1)}_s+\kappa_s,\phi^{(2)}_s+B_s^H)-b(\phi_s)\big)}{s^{\a}}\nonumber\\&+\a s^{\a}\int_{0}^{s}\frac{s^{-\a}\big(b(\phi^{(1)}_s+\kappa_s,\phi^{(2)}_s+B_s^H)-b(\phi_s)\big)}{(s-r)^{\a+1}}-\frac{r^{-\a}\big(b(\phi^{(1)}_r+\kappa_r,\phi^{(2)}_r+B_r^H)-b(\phi_r)\big)}{(s-r)^{\a+1}}dr\Bigg|\nonumber\\&\leq J_{31}+J_{32}+J_{33},\nonumber
\end{align}
where
\begin{align}
	J_{31}&=\frac{1}{\G(1-\a)}\Big|\frac{b(\phi^{(1)}_s+\kappa_s,\phi^{(2)}_s+B_s^H)-b(\phi_s)\big)}{s^{\a}}\Big|\nonumber\\J_{32}&=\frac{\a}{\G(1-\a)}\Bigg|\int_{0}^{s}\frac{\big(b(\phi^{(1)}_s+\kappa_s,\phi^{(2)}_s+B_s^H)-b(\phi_s)\big)-\big(b(\phi^{(1)}_r+\kappa_r,\phi^{(2)}_r+B_r^H)-b(\phi_r)\big)}{(s-r)^{\a+1}}\Bigg|\nonumber\\J_{33}&=\frac{\a s^{\a}}{\G(1-\a)}\Bigg|\int_{0}^{s}\frac{s^{-\a}-r^{-\a}}{(s-r)^{\a+1}}\big(b(\phi^{(1)}_r+\kappa_r,\phi^{(2)}_r+B_r^H)-b(\phi_r)\big)dr\Bigg|.\nonumber
\end{align}
So under the condition $\|B^H\|_{\beta}<\varepsilon$, from \eqref{holder of kappa} and using the fact that $b$ is Lipschitz continuous and bounded with constant $L_2$, we have that ($\b>\a$)
\begin{align}\label{J_{31}}
	J_{31}\leq \frac{(C(L_1)+1)L_2}{\G(1-\a)}\varepsilon.
\end{align}
On the other hand, from ($v=\frac{r}{s}$)
\begin{align}\label{integral transform}
	\int_{0}^{s}\frac{r^{-\a}-s^{-\a}}{(s-r)^{\a+1}}dr=s^{-2\a}\int_{0}^{1}\frac{v^{-\a}-1}{(1-v)^{\a+1}}dv:=C_{\a}s^{-2\a},
\end{align}
where
$C_{\a}$ is a constant depending on $\a$. So by \eqref{holder of kappa} we have
\begin{align}\label{J_{33}}
	J_{33}\leq& \frac{\a s^{\a}L_2}{\G(1-\a)}\int_{0}^{s}\frac{r^{-\a}-s^{-\a}}{(s-r)^{\a+1}}\Big(\big|B_r^H\big|+|\kappa_r|\Big)dr\nonumber\\\leq& \frac{C_{\a}\a s^{\b-\a}L_2}{\G(1-\a)}(1+C(L_1))\|B^H\|_{\b}\leq C(\a,L_1,L_2)\varepsilon.
\end{align}
It remains to study the limit behavior of the term $J_{32}$. To begin with, we give an integral equality.
\begin{align}
	&\Big(b(\phi^{(1)}_s+\kappa_s,\phi^{(2)}_s+B_s^H)-b(\phi_s)\Big)-\Big(b(\phi^{(1)}_r+\kappa_r,\phi^{(2)}_r+B_r^H)-b(\phi_r)\Big)\nonumber\\=&\Big(b(\phi^{(1)}_s+\kappa_s,\phi^{(2)}_s+B_s^H)-b(\phi^{(1)}_s, \phi^{(2)}_s+B^H_s)\Big)+\Big(b(\phi^{(1)}_s,\phi^{(2)}_s+B_s^H)-b(\phi^{(1)}_s, \phi^{(2)}_s)\Big)\nonumber\\-&\Big(b(\phi^{(1)}_r+\kappa_r,\phi^{(2)}_r+B_r^H)-b(\phi^{(1)}_r, \phi^{(2)}_r+B^H_r)\Big)-\Big(b(\phi^{(1)}_r,\phi^{(2)}_r+B_r^H)-b(\phi^{(1)}_r, \phi^{(2)}_r)\Big)\nonumber\\=&\int_{0}^{1}b_x(\lambda \kappa_s+\phi^{(1)}_s,\phi^{(2)}_s+B_s^H)d\lambda  \cdot\kappa_s+\int_{0}^{1}b_y(\phi^{(1)}_s,\phi^{(2)}_s+\mu B_s^H)d\mu \cdot B_s^H\nonumber\\-&\int_{0}^{1}b_x(\lambda \kappa_r+\phi^{(1)}_r,\phi^{(2)}_r+B_r^H)d\lambda\cdot  \kappa_r-\int_{0}^{1}b_y(\phi^{(1)}_r,\phi^{(2)}_r+\mu B_r^H)d\mu \cdot B_r^H\nonumber\nonumber\\=&\int_{0}^{1}\Big(b_x(\lambda \kappa_s+\phi^{(1)}_s,\phi^{(2)}_s+B_s^H)-b_x(\lambda \kappa_r+\phi^{(1)}_r,\phi^{(2)}_r+B_r^H)\Big)d\lambda \cdot \kappa_s+\int_{0}^{1}b_x(\lambda \kappa_r+\phi^{(1)}_r,\phi^{(2)}_r+B_r^H)d\lambda\cdot\Big(\kappa_s-\kappa_r\Big)\nonumber\\+&\int_{0}^{1}\Big(b_y(\phi^{(1)}_s,\phi^{(2)}_s+\mu B_s^H)-b_y(\phi^{(1)}_r,\phi^{(2)}_r+\mu B_r^H)\Big)d\lambda\cdot B^H_s+\int_{0}^{1}b_y(\phi^{(1)}_r,\phi^{(2)}_r+\mu B_r^H)d\lambda\cdot\Big(B^H_s-B^H_r\Big)\nonumber
\end{align}
Hence, using that $b_x,b_y$ are bounded and Lipschitz with constant $L_3, L_4$, and $\phi=(\phi^1,\phi^2)$ is $H$-H$\mathrm{\ddot{o}}$lder continuous yields
\begin{align}
	&\Bigg|\Big(b(\phi^{(1)}_s+\kappa_s,\phi^{(2)}_s+B_s^H)-b(\phi_s)\Big)-\Big(b(\phi^{(1)}_r+\kappa_r,\phi^{(2)}_s+B_r^H)-b(\phi_r)\Big)\Bigg|\nonumber\\=&\Bigg|\int_{0}^{1}\Big(b_x(\lambda \kappa_s+\phi^{(1)}_s,\phi^{(2)}_s+B_s^H)-b_x(\lambda \kappa_r+\phi^{(1)}_r,\phi^{(2)}_r+B_r^H)\Big)d\lambda\cdot \kappa_s\nonumber\\&+\int_{0}^{1}b_x(\lambda \kappa_r+\phi^{(1)}_r,\phi^{(2)}_r+B_r^H)d\lambda\cdot\Big(\kappa_s-\kappa_r\Big)\nonumber\\&+\int_{0}^{1}\Big(b_y(\phi^{(1)}_s,\phi^{(2)}_s+\mu B_s^H)-b_y(\phi^{(1)}_r,\phi^{(2)}_r+\mu B_r^H)\Big)d\mu\cdot B^H_s+\int_{0}^{1}b_y(\phi^{(1)}_r,\phi^{(2)}_r+\mu B_r^H)d\mu\cdot\Big(B^H_s-B^H_r\Big)\nonumber\Bigg|\\&\leq L_3\Big(|\phi_s^{(1)}-\phi_r^{(1)}|+|\phi_s^{(2)}-\phi_r^{(2)}|+\frac{1}{2}|\kappa_s-\kappa_r|+|B_s^H-B_r^H|\Big)|\kappa_s|+\|b_x\|_{\infty}|\kappa_s-\kappa_r|\nonumber\\&+L_4\Big(|\phi_s^{(1)}-\phi_r^{(1)}|+|\phi_s^{(2)}-\phi_r^{(2)}|+\frac{1}{2}|B_s^H-B_r^H|\Big)|B_s^H|+\|b_y\|_{\infty}|B_s^H-B_r^H|\nonumber\\&\leq L_3\Big(\|\phi\|_{H}(s-r)^H+\frac{1}{2}\|\kappa\|_{\beta}(s-r)^{\b}+\|B^H\|_{\b}(s-r)^{\b}\Big)\|\kappa\|_{\b}s^{\b}+\|b_x\|_{\infty}\|\kappa\|_{\b}(s-r)^{\b}\nonumber\\&+ L_4\Big(\|\phi\|_H(s-r)^H+\frac{1}{2}\|B^H\|_{\b}(s-r)^{\b}\Big)\|B^H\|_{\b}s^{\b}+\|b_y\|_{\infty}\|B^H\|_{\b}(s-r)^{\b}\nonumber\\&\leq L_3C(L_1)\Big(\|\phi\|_{H}(s-r)^H+(\frac{C(L_1)}{2}+1)\|B^H\|_{\beta}(s-r)^{\b}\Big)\|B^H\|_{\b}s^{\b}+C(L_1)\|b_x\|_{\infty}\|B^H\|_{\b}(s-r)^{\b}\nonumber\\&+ L_4\Big(\|\phi\|_H(s-r)^H+\frac{1}{2}\|B^H\|_{\b}(s-r)^{\b}\Big)\|B^H\|_{\b}s^{\b}+\|b_y\|_{\infty}\|B^H\|_{\b}(s-r)^{\b}\nonumber.
\end{align}
where \eqref{holder of kappa} is used. Therefore we have ($u=\frac{r}{s}$)
\begin{align}\label{J_{32}}
	J_{32}&=\frac{\a}{\G(1-\a)}\Bigg|\int_{0}^{s}\frac{\big(b(\phi^{(1)}_s+\kappa_s,\phi^{(2)}_s+B_s^H)-b(\phi_s)\big)-\big(b(\phi^{(1)}_r+\kappa_r,\phi^{(2)}_r+B_r^H)-b(\phi_r)\big)}{(s-r)^{\a+1}}dr\Bigg|\nonumber\\&\leq \frac{\a}{\G(1-\a)}\Bigg|\int_{0}^{s}\frac{L_3C(L_1)\Big(\|\phi\|_{H}(s-r)^H+(\frac{C(L_1)}{2}+1)\|B^H\|_{\beta}(s-r)^{\b}\Big)\|B^H\|_{\b}s^{\b}}{(s-r)^{\a+1}}\nonumber\\&+\frac{C(L_1)\|b_x\|_{\infty}\|B^H\|_{\b}(s-r)^{\b}+L_4\Big(\|\phi\|_H(s-r)^H+\frac{1}{2}\|B^H\|_{\b}(s-r)^{\b}\Big)\|B^H\|_{\b}s^{\b}}{(s-r)^{\a+1}}\nonumber\\&+\frac{\|b_y\|_{\infty}\|B^H\|_{\b}(s-r)^{\b}}{(s-r)^{\a+1}}dr\Bigg|\nonumber\\&\leq C(\a,L_1,L_3, \phi)\big|\int_{0}^{s}\frac{(s-r)^H s^{\b}}{(s-r)^{\a+1}}dr\big|\|B^H\|_{\b}+C(\a,L_1,L_3)\big|\int_{0}^{s}\frac{(s-r)^{\b}s^{\b}}{(s-r)^{\a+1}}dr\big|\|B^H\|_{\b}\nonumber\\&+C(\a,L_1,b_x)\big|\int_{0}^{s}\frac{(s-r)^{\b}}{(s-r)^{\a+1}}dr\big|\|B^H\|_{\b}+C(\a,L_4,\phi)\big|\int_{0}^{s}\frac{(s-r)^H s^{\b}}{(s-r)^{\a+1}}dr\big|\|B^H\|_{\b}\nonumber\\&+C(\a,L_4)\big|\int_{0}^{s}\frac{(s-r)^{\b}s^{\b}}{(s-r)^{\a+1}}dr\big|\|B^H\|_{\b}+C(\a,b_y)\big|\int_{0}^{s}\frac{(s-r)^{\b}}{(s-r)^{\a+1}}dr\big|\|B^H\|_{\b}\nonumber\\&\leq\Bigg[ C(\a,L_1,L_3,\phi)s^{H+\b-\a}\int_{0}^{1}(1-u)^{H-\a-1}du+C(\a,L_1,L_3)s^{2\b-\a}\int_{0}^{1}(1-u)^{\b-\a-1}du\nonumber\\&+C(\a,L_1,b_x)s^{\b-\a}\int_{0}^{1}(1-u)^{\b-\a-1}du+C(\a,L_4,\phi)s^{H+\b-\a}\int_{0}^{1}(1-u)^{H-\a-1}du\nonumber\\&+C(\a,L_4)s^{2\b-\a}\int_{0}^{1}(1-u)^{\b-\a-1}du+C(\a,b_y)s^{\b-\a}\int_{0}^{1}(1-u)^{\b-\a-1}du\Bigg]\|B^H\|_{\b}\nonumber\\&\leq C(\a,\b,L_1,L_3,L_4,\phi,b_x,b_y)\varepsilon.
\end{align}
Hence, it follows from estimates \eqref{J_{31}}, \eqref{J_{33}} and \eqref{J_{32}} that
\begin{align}\label{the following need}
	\Big|s^{\alpha}\big(D^{\alpha}_{0+}u^{-\alpha}\big(b(\phi^{(1)}_u+\kappa_u,\phi^{(2)}_u+B_u^H)-b(\phi_u)\big)\big)(s)\Big|\leq C(L_1,L_2,L_3,L_4,\a,\b,\phi,b_x,b_y)\varepsilon.
\end{align}
Therefore,
\begin{align}\label{J_3}
	\limsup _{\varepsilon \rightarrow 0}\ E(\exp(cJ_3)|\|B^H\|_{\beta}<\varepsilon)\leq1,
\end{align}
for every real number $c$.

$\clubsuit \ \text{Term}$ \ $J_4$\\
By inequality $|a^2-b^2|\leq (a-b)^2+2|(a-b)b|$, we have
\begin{align}
	|J_4|&=\frac{1}{2}\int_{0}^{1}\Bigg|\Big(s^{\alpha}\big(D^{\alpha}_{0+}u^{\alpha}b(\phi_u)\big)(s)\Big)^2-\Big(s^{\a}\big(D^{-\alpha}_{0+}u^{\alpha}b(\phi^{(1)}_u+\kappa_u,\phi^{(2)}_u+B_u^H)\big)(s)\Big)^2\Bigg|ds\nonumber\\&\leq \frac{1}{2}\int_{0}^{1}\Big(s^{\a}\big(D^{\alpha}_{0+}u^{\alpha}(b(\phi^{(1)}_u+\kappa_u,\phi^{(2)}_u+B_u^H)-b(\phi_u))\big)(s)\Big)^2ds\nonumber\\&+\int_{0}^{1}\Big|\Big(s^{\a}\big(D^{\alpha}_{0+}u^{\alpha}(b(\phi^{(1)}_u+\kappa_u,\phi^{(2)}_u+B_u^H)-b(\phi_u))\big)(s)\Big)s^{\a}D^{\a}_{0+}s^{-\a}b(\phi_s)\Big|ds\nonumber
\end{align}
Using \eqref{the following need} it is easy to see that
$$|J_4|\leq C(L_1,L_2,L_3,L_4,\a,\b,\phi,b_x,b_y)\varepsilon.$$
As a consequence,
\begin{align}\label{J_4}
	\limsup _{\varepsilon \rightarrow 0}\ E(\exp(cJ_4)|\|B^H\|_{\beta}<\varepsilon)\leq1,
\end{align}
for every real number $c$.

$\clubsuit \ \text{Term}$ \ $J_1$\\
Applying classical Taylor expansion to $b(\phi^{(1)}_s+\kappa_s,B_s^H+\phi^{(2)}_s)$ at $\phi$ we have
\begin{align}
	b(\phi^{(1)}_s+\kappa_s,B_s^H+\phi^{(2)}_s)=b(\phi_s)+b_x(\phi_s)\kappa_s+b_y(\phi_s)B_s^H+R_s^{x,y},\nonumber
\end{align}
where $R^{x,y}$ denotes the remainder term. If $\|B^H\|_{\b}\leq \varepsilon$, by Young's inequality we have that
\begin{align}\label{estimate of Remainder term}
	\|R^{x,y}\|_{\infty}\leq C(L_1)\varepsilon^2.
\end{align}
Hence, we can rewrite the term $J_1$.
\begin{align}\label{the part of J_1}
	J_1&=\int_{0}^{1}s^{\alpha}\big(D^{\alpha}_{0+}u^{-\alpha}b(\phi^{(1)}_u+\kappa_u,\phi^{(2)}_u+B_u^H)\big)(s)dW_s\nonumber\\&=\int_{0}^{1}s^{\alpha}\Big(D^{\alpha}_{0+}u^{-\alpha}\big(b(\phi_u)+b_x(\phi_u)\kappa_u+b_y(\phi_u)B_u^H+R_u^{x,y}\big)\Big)(s)dW_s\nonumber\\&:=J_{11}+J_{12}+J_{13}+J_{14},\
\end{align}
Applying Theorem \ref{no random function} to $f=cs^{\a}\big(D^{\a}_{0+}u^{-\a}b(\phi_u)\big)(s)$ and Lemma \ref{norm}, we have
\begin{align}\label{J_{11}}
	\limsup _{\varepsilon \rightarrow 0}\ E(\exp(cJ_{11})|\|B^H\|_{\beta}<\varepsilon)\leq1,
\end{align}
for every real number $c$.\\
We now deal with the term $J_{12}$. By using the fact $b_x$ is bounded, the definition of $\kappa$ with $\|\kappa\|_{\b} \leq C(L_1)\|B^H\|_{\b}$ and the formula for fractional integration by parts $(2.69)$ in \cite{SKM93}, so we have
\begin{align}
	J_{12}=&\int_{0}^{1}s^{\alpha}\Big(D^{\alpha}_{0+}u^{-\alpha}\big(b_x(\phi_u)\kappa_u\big)\Big)(s)dW_s\nonumber\\=&\int_{0}^{1}s^{-\a}b_x(\phi_s)\kappa_s \big(D^{\a}_{1-}u^{\a}\big)(s)dW_s\nonumber\\=&\int_{0}^{1}s^{-\a}b_x(\phi_s) \big(D^{\a}_{1-}u^{\a}\big)(s)\int_{0}^{s}\big(\sigma(\tilde{X}_r,\tilde{Y}_r)-\sigma(\phi_r)\big)drdW_s\nonumber\\=&\int_{0}^{1}\big(\sigma(\tilde{X}_r,\tilde{Y}_r)-\sigma(\phi_r)\big)\int_{r}^{1}s^{-\a}b_x(\phi_s) \big(D^{\a}_{1-}u^{\a}\big)(s)dW_sdr,&\nonumber\\\leq&\max\{\int_{0}^{1}C\varepsilon\int_{r}^{1}s^{-\a}b_x(\phi_s) \big(D^{\a}_{1-}u^{\a}\big)(s)dW_sdr, -\int_{0}^{1}C\varepsilon\int_{r}^{1}s^{-\a}b_x(\phi_s) \big(D^{\a}_{1-}u^{\a}\big)(s)dW_sdr\},\nonumber
\end{align}
Then we further obtain
\begin{align}
	0\leq |J_{12}|&\leq \Big|\int_{0}^{1}C\varepsilon\int_{r}^{1}s^{-\a}b_x(\phi_s) \big(D^{\a}_{1-}u^{\a}\big)(s)dW_sdr\Big|\nonumber\\&\leq\Big|\int_{0}^{1}C\varepsilon s^{1-\a}b_x(\phi_s) \big(D^{\a}_{1-}u^{\a}\big)(s) dW_s\Big|\nonumber
\end{align}
Applying Theorem \ref{no random function 2} to $h=C\varepsilon s^{1-\a}b_x(\phi_s) \big(D^{\a}_{1-}u^{\a}\big)(s)$, we have
\begin{align}\label{J_{12}}
	\limsup _{\varepsilon \rightarrow 0}\ E(\exp(cJ_{12})|\|B^H\|<\varepsilon)\leq1,
\end{align}
for every real number $c$.
\\
In order to study the limit behavior of the conditional exponential moments of the term $J_{13}$. We will express $J_{13}$ as a double stochastic integral with respect to $W$ based on the Weyl representation of the fractional derivative and the integral representation of fractional Brownian motion $B^H$, so we obtain
\begin{align}
	J_{12}=&\int_{0}^{1}s^{\alpha}\Big(D^{\alpha}_{0+}u^{-\alpha}\big(b_y(\phi_u)B^H_u\big)\Big)(s)dW_s\nonumber\\=&\frac{1}{\G(1-\a)}\int_{0}^{1}\Big(s^{-\a}b_y(\phi_s)B_s^H+\a s^{\a}\int_{0}^{s}\frac{s^{-\a}b_y(\phi_s)B_s^H-u^{-\a}b_y(\phi_u)B_u^H}{(s-u)^{\a+1}}du\Big)dW_s\nonumber\\=&\frac{1}{\G(1-\a)}\int_{0}^{1}\int_{0}^{1}s^{-\a}b_y(\phi_s)K^H(s,r)\mathbb{I}_{r\leq s}dW_r dW_s+\a\int_{0}^{1}s^{\a}\nonumber\\&\cdot \int_{0}^{s}\frac{s^{-\a}b_y(\phi_s)\int_{0}^{s}K^H(s,r)dW_r-u^{-\a}b_y(\phi_u)\int_{0}^{u}K^H(u,r)dW_r}{(s-u)^{\a+1}}dudW_s\nonumber\\=&\int_{0}^{1}\int_{0}^{1}f(s,r)dW_rdW_s=\int_{0}^{1}\int_{0}^{1}\tilde{f}(s,r)dW_rdW_s,\nonumber
\end{align}
where $\tilde{f}$ is the symmetrization of the function ($\tilde f(s,r)=\tilde f(r,s)$)
\begin{align}
	f(s,r)=&\frac{1}{\G(1-\a)}\Big(s^{-\a}b_y(\phi_s)K^H(s,r)\mathbb{I}_{r\leq s}+\a s^{\a}\nonumber\\&\cdot \int_{0}^{s}\frac{s^{-\a}b_y(\phi_s)K^H(s,r)\mathbb{I}_{r\leq s}-u^{-\a}b_y(\phi_u)K^H(u,r)\mathbb{I}_{r\leq u}}{(s-u)^{\a+1}}du\Big)\nonumber
\end{align}
By Lemma \ref{trace class of regular case} the operator $K(\tilde{f})$ is nuclear. So the trace of this operator can be obtained as
$$Tr\tilde{f}=\int_{0}^{1}\tilde{f}(s,s)ds=\frac{1}{2}\int_{0}^{1}f(s,s)ds.$$
In order to compute the integral $\int_{0}^{1}f(s,s)ds.$ Let us rewrite $f(s,r)$ by $\int_{0}^{u}=\int_{0}^{r}+\int_{r}^{u}$ and $K^H(u,u)=0$,
\begin{align}
	f(s,r)=&\frac{1}{\G(1-\a)}\Big(s^{-\a}b_y(\phi_s)K^H(s,r)\mathbb{I}_{r\leq s}+\a s^{\a}\nonumber\\&\cdot \int_{0}^{u}\frac{s^{-\a}b_y(\phi_s)K^H(s,r)\mathbb{I}_{r\leq s}-u^{-\a}b_y(\phi_u)K^H(u,r)\mathbb{I}_{r\leq u}}{(s-u)^{\a+1}}du\Big)\nonumber\\=&\frac{1}{\G(1-\a)}s^{-\a}b_y(\phi_s)K^H(s,r)\mathbb{I}_{r\leq s}\nonumber\\+&\frac{\a s^{\a}}{\G(1-\a)}\int_{0}^{r}\frac{s^{-\a}b_y(\phi_s)K^H(s,r)\mathbb{I}_{r\leq s}-u^{-\a}b_y(\phi_u)K^H(u,r)\mathbb{I}_{r\leq u}}{(s-u)^{\a+1}}du\nonumber\\+&\frac{\a s^{\a}}{\G(1-\a)}\int_{r}^{s}\frac{s^{-\a}b_y(\phi_s)K^H(s,r)\mathbb{I}_{r\leq s}-u^{-\a}b_y(\phi_u)K^H(u,r)\mathbb{I}_{r\leq u}}{(s-u)^{\a+1}}du\nonumber\\ =&\frac{1}{\G(1-\a)}s^{-\a}b_y(\phi_s)K^H(s,r)\mathbb{I}_{r\leq s}+\frac{\a s^{\a}}{\G(1-\a)}\int_{0}^{r}\frac{s^{-\a}b_y(\phi_s)K^H(s,r)\mathbb{I}_{r\leq s}}{(s-u)^{\a+1}}du\nonumber\\+& \frac{\a s^{\a}}{\G(1-\a)}\int_{r}^{s}\frac{s^{-\a}b_y(\phi_s)K^H(s,r)\mathbb{I}_{r\leq s}-u^{-\a}b_y(\phi_u)K^H(u,r)\mathbb{I}_{r\leq u}}{(s-u)^{\a+1}}du\nonumber\\:=&f_1(s,r)+f_2(s,r)+f_3(s,r).\nonumber
\end{align}
It is clear that $f_1(s,s)=0$ when $s=r$ due to $K^H(s,s)=0$. Since $H\geq \frac{1}{2}$ the kernel $K^H$ can be written as
\begin{align}\label{the representation of K^H}
	K^H(s,r)=c_H\a r^{-\a}\int_{r}^{s}(\theta-r)^{\a-1}\theta^{\a}d\theta.
\end{align}
The change of variable $w=\frac{\theta-r}{s-r}$ yields
\begin{align}
	f_2(s,r)=&\frac{\a s^{\a}}{\G(1-\a)}\int_{0}^{r}\frac{s^{-\a}b_y(\phi_s)K^H(s,r)\mathbb{I}_{r\leq s}}{(s-u)^{\a+1}}du\nonumber\\=&\frac{1}{\G(1-\a)}b_y(\phi_s)K^H(s,r)\big((s-r)^{-\a}-s^{-\a}\big)\nonumber\\=&\frac{\a c_H}{\G(1-\a)}b_y(\phi_s)\big((s-r)^{-\a}-s^{-\a}\big)r^{-\a}\int_{r}^{s}(\theta-r)^{\a-1}\theta^{\a}d\theta\nonumber\\=&\frac{\a c_H}{\G(1-\a)}b_y(\phi_s)\big((s-r)^{-\a}-s^{-\a}\big)(s-r)^{\a}\int_{0}^{1}w^{\a-1}\big((s-r)w+r\big)^{\a}dw\nonumber,
\end{align}
and hence, we have
\begin{align}
	f_2 (s,s)=\frac{\a c_H}{\G(1-\a)}b_y(\phi_{s})\int_{0}^{1}w^{\a-1}dw=\frac{ c_H}{\G(1-\a)}b_y(\phi_{s}).\nonumber
\end{align}
On the other hand,
\begin{align}
	f_3(s,r)=&\frac{\a s^{\a}}{\G(1-\a)}\int_{r}^{s}\frac{s^{-\a}b_y(\phi_s)K^H(s,r)\mathbb{I}_{r\leq s}-u^{-\a}b_y(\phi_u)K^H(u,r)\mathbb{I}_{r\leq u}}{(s-u)^{\a+1}}du\nonumber\\:=&f_{31}(s,r)+f_{32}(s,r)+f_{33}(s,r),\nonumber
\end{align}
where $(r\leq u\leq s)$
\begin{align}
	f_{31}(s,r)&=\frac{\a s^{\a}}{\G(1-\a)}b_y(\phi_s)K^H(s,r)\int_{r}^{s}\frac{s^{-\a}-u^{-\a}}{(s-u)^{\a+1}}du,\nonumber\\f_{32}(s,r)&=\frac{\a s^{\a}}{\G(1-\a)}K^H(s,r)\int_{r}^{s}\frac{b_y(\phi_s)-b_y(\phi_u)}{(s-u)^{\a+1}}u^{-\a}du,\nonumber\\f_{33}(s,r)&=\frac{\a s^{\a}}{\G(1-\a)}\int_{r}^{s}\frac{K^H(s,r)-K^H(u,r)}{(s-u)^{\a+1}}u^{-\a}b_y(\phi_u)du.\nonumber
\end{align}
Recall inequality $|s^{-\a}-u^{-\a}|\leq\a r^{-\a-1}(s-u)$, so we have
$$|f_{31}(s,r)|\leq\frac{\a^2 s^{\a}}{(1-\a)\G(1-\a)}r^{-\a-1}(s-r)^{1-\a}|b_y(\phi_s)K^H(s,r)|,$$
and hence $f_{31}(s,s)=0$.
Since $\phi$ is $H$-H$\mathrm{\ddot{o}}$lder continuous and $b_y$ is Lipschitz continuous with constant $L_3$ we have
\begin{align}
	|f_{32}(s,r)|&=\frac{\a s^{\a}}{\G(1-\a)}\Big| K^H(s,r)\int_{r}^{s}\frac{b_y(\phi_s)-b_y(\phi_u)}{(s-u)^{\a+1}}u^{-\a}du\Big|	\nonumber\\&\leq C(L_3)\frac{\a s^{\a}}{\G(1-\a)}\Big| K^H(s,r)\int_{r}^{s}(s-u)^{-\frac{1}{2}}u^{-\a}du\Big|,\nonumber
\end{align}
Which implies that $f_{32}(s,s)=0$.

Using the expression \eqref{the representation of K^H} of the kernel $K^H$, we have
\begin{align}
	f_{33}(s,r)&=\frac{\a s^{\a}}{\G(1-\a)}\int_{r}^{s}\frac{K^H(s,r)-K^H(u,r)}{(s-u)^{\a+1}}u^{-\a}b_y(\phi_u)du\nonumber\\&=\frac{c_H\a^2 s^{\a}r^{-\a}}{\G(1-\a)}\int_{r}^{s}\frac{\int_{u}^{s}(\theta-r)^{\a-1}\theta^{\a}d\theta}{(s-u)^{\a+1}}u^{-\a}b_y(\phi_u)du\nonumber\\&=\frac{c_H\a^2 s^{\a}r^{-\a}}{\G(1-\a)}(s-r)^{\a}\int_{r}^{s}\frac{\int_{\frac{u-r}{s-r}}^{1}m^{\a-1}\big((s-r)m+r\big)^{\a}dm}{(s-u)^{\a+1}}u^{-\a}b_y(\phi_u)du\nonumber\\&=\frac{c_H\a^2 s^{\a}r^{-\a}}{\G(1-\a)}\int_{0}^{1}\frac{\int_{n}^{1}m^{\a-1}\big((s-r)m+r\big)^{\a}dm}{(1-n)^{\a+1}}\big((s-r)n+r\big)^{-\a}b_y(\phi_{(s-r)n+r})dn,\nonumber
\end{align}
where the last two equality has been obtained with the change of variables $m=\frac{\theta-r}{s-r}$ and $n=\frac{u-r}{s-r}$.
So as variable $s=r$ we easily get
\begin{align}
	f_{33}(s,s)&=\frac{c_H\a^2 }{\G(1-\a)}b_y(\phi_{s})\int_{0}^{1}\frac{\int_{n}^{1}m^{\a-1}s^{\a}dm}{(1-n)^{\a+1}}s^{-\a}dn\nonumber\\&=\frac{c_H\a }{\G(1-\a)}b_y(\phi_{s})\int_{0}^{1}\frac{1-n^{\a}}{(1-n)^{\a+1}}dn\nonumber\\&=\frac{c_H\a }{\G(1-\a)}b_y(\phi_{s})\frac{\big(\G(1+\a)\G(1-\a)-1\big)}{\a}\nonumber\\&=c_H\G(\a+1)b_y(\phi_{s})-\frac{c_H}{\G(1-\a)}b_y(\phi_{s}).\nonumber
\end{align}
As a consequence, we have
\begin{align}
	Tr(\tilde{f})&=\frac{1}{2}\int_{0}^{1}f(s,s)ds=\frac{1}{2}\int_{0}^{1}f_{1}(s,s)+f_{2}(s,s)+f_3(s,s)ds\nonumber\\&=\frac{1}{2}\int_{0}^{1}\frac{c_H}{\G(1-\a)}b_y(\phi_{s})ds+\frac{1}{2}\int_{0}^{1}f_{31}(s,s)+f_{32}(s,s)+f_{33}(s,s)ds\nonumber\\&=\frac{c_H\G(1+\a)}{2}\int_{0}^{1}b_y(\phi_{s})ds=\frac{d_H}{2}\int_{0}^{1}b_y(\phi_{s})ds.\nonumber
\end{align}
To summarized what we have proved, Lemma \ref{key lemma} and Lemma \ref{norm} give us
$$
\limsup _{\varepsilon \rightarrow 0}\ E(\exp(J_{13})|\|B^H\|<\varepsilon)=\exp\Big(-\frac{d_H}{2}\int_{0}^{1}b_y(\phi_{s})ds\Big).
$$
Finally, it only remains to study the limit behavior of the term $J_{14}$. For any $c\in \mathbb{R}$ and $\delta>0$ we can write
\begin{align}
	E&\Big(\exp(cJ_{14})|\|B^H\|\leq \varepsilon\Big)\nonumber\\&\leq e^{\delta}+\int_{\delta}^{\infty}e^{\xi}P(|cJ_{14}|>\xi|\|B^H\|_{\b}\leq \varepsilon)d \xi \nonumber\\&+e^{\delta}\mathbb{P}(|cJ_{14}|>\delta |\|B^H\|_{\b}\leq \varepsilon).\nonumber
\end{align}
Define the martingale $M_t=c\int_{0}^{t}s^{-\a}\Big(D^{\alpha}_{0+}u^{-\alpha}R_u^{x,y}\Big)(s)dW_s$. In order to estimate whose quadratic variations we make use of the following expression of the residual term
\begin{align}
	R_s^{x,y}&=b(\phi^{(1)}_s+\kappa_s,B_s^H+\phi^{(2)}_s)-b(\phi_s)-b_x(\phi_s)\kappa_s-b_y(\phi_s)B_s^H\nonumber\\&=b(\phi^{(1)}_s+\kappa_s,B_s^H+\phi^{(2)}_s)-b(\phi_s^{(1)},\phi^{(2)}_s+B_s^H)+b(\phi_s^{(1)},\phi^{(2)}_s+B_s^H)-b(\phi_s)\nonumber\\&-\int_{0}^{1}b_x(\phi_s)d\lambda\kappa_s-\int_{0}^{1}b_y(\phi_s)d\mu B_s^H\nonumber\\&=\int_{0}^{1}\big(b_x(\phi_s^{(1)}+\lambda\kappa_s,\phi_s^{(2)}+B_s^H)-b_x(\phi_s)\big)d\lambda \kappa_s+\int_{0}^{1}\big(b_y(\phi_s^{(1)},\phi_s^{(2)}+\mu B_s^H)-b_y(\phi_s)\big)d\mu B^H_s\nonumber\\&=\int_{0}^{1}\big(b_x(\phi_s^{(1)}+\lambda\kappa_s,\phi_s^{(2)}+B_s^H)-b_x(\phi^{(1)}_s,\phi^{(2)}_s+B_s^H)\big)d\lambda \kappa_s+\int_{0}^{1}\big(b_x(\phi_s^{(1)},\phi_s^{(2)}+B_s^H)-b_x(\phi^{(1)}_s,\phi^{(2)}_s)\big)d\lambda \kappa_s\nonumber\\&+\int_{0}^{1}\big(b_y(\phi_s^{(1)},\phi_s^{(2)}+\mu B_s^H)-b_y(\phi^{(1)}_s,\phi^{(2)}_s)\big)d\mu B^H_s\nonumber\\&=\int_{0}^{1}\int_{0}^{\lambda}b_{xx}(\phi_s^{(1)}+\theta\kappa_s,\phi_s^{(2)}+B_s^H)d\theta d\lambda (\kappa_s)^2+\int_{0}^{1}\int_{0}^{1}b_{xy}(\phi_s^{(1)},\phi_s^{(2)}+\pi B_s^H)d\pi d\lambda (B_s^H\kappa_s)\nonumber\\&+\int_{0}^{1}\int_{0}^{\lambda}b_{yy}(\phi_s^{(1)},\phi_s^{(2)}+\nu B_s^H)d\nu d\mu (B^H_s)^2\nonumber.
\end{align}
Using that $b_{xx},b_{xy}, b_{yy}$ are Lipschitz with constant $L_5, L_6,L_7$, we have
\begin{align}
	&R_s^{x,y}-R_r^{x,y}\nonumber\\&=\int_{0}^{1}\int_{0}^{\lambda}\Big[b_{xx}(\phi_s^{(1)}+\theta\kappa_s,\phi_s^{(2)}+B_s^H)-b_{xx}(\phi_r^{(1)}+\theta\kappa_r,\phi_r^{(2)}+B_r^H)\Big]d\theta d\lambda (\kappa_s)^2\nonumber\\&+\int_{0}^{1}\int_{0}^{1}\Big[b_{xy}(\phi_s^{(1)},\phi_s^{(2)}+\pi B_s^H)-b_{xy}(\phi_r^{(1)},\phi_r^{(2)}+\pi B_r^H)\Big]d\pi d\lambda (B_s^H\kappa_s)\nonumber\\&+\int_{0}^{1}\int_{0}^{\lambda}\Big[b_{yy}(\phi_s^{(1)},\phi_s^{(2)}+\nu B_s^H)-b_{yy}(\phi_r^{(1)},\phi_r^{(2)}+\nu B_r^H)\Big]d\nu d\mu (B^H_s)^2\nonumber\\&+\int_{0}^{1}\int_{0}^{\lambda}b_{xx}(\phi_r^{(1)}+\theta\kappa_r,\phi_r^{(2)}+B_r^H)d\theta d\lambda ((\kappa_s)^2-(\kappa_r)^2)+\int_{0}^{1}\int_{0}^{1}b_{xy}(\phi_r^{(1)},\phi_r^{(2)}+\pi B_r^H)d\pi d\lambda (B_s^H\kappa_s-B_r^H\kappa_r)\nonumber\\&+\int_{0}^{1}\int_{0}^{\lambda}b_{yy}(\phi_r^{(1)},\phi_r^{(2)}+\nu B_r^H)d\nu d\mu ((B^H_s)^2-(B^H_r)^2)\nonumber\\&\leq\int_{0}^{1}\int_{0}^{\lambda}\Big|b_{xx}(\phi_s^{(1)}+\theta\kappa_s,\phi_s^{(2)}+B_s^H)-b_{xx}(\phi_r^{(1)}+\theta\kappa_r,\phi_r^{(2)}+B_r^H)\Big|d\theta d\lambda (\kappa_s)^2\nonumber\\&+\int_{0}^{1}\int_{0}^{1}\Big|b_{xy}(\phi_s^{(1)},\phi_s^{(2)}+\pi B_s^H)-b_{xy}(\phi_r^{(1)},\phi_r^{(2)}+\pi B_r^H)\Big|d\pi d\lambda |B_s^H\kappa_s|\nonumber\\&+\int_{0}^{1}\int_{0}^{\lambda}\Big|b_{yy}(\phi_s^{(1)},\phi_s^{(2)}+\nu B_s^H)-b_{yy}(\phi_r^{(1)},\phi_r^{(2)}+\nu B_r^H)\Big|d\nu d\mu (B^H_s)^2\nonumber\\&+\int_{0}^{1}\int_{0}^{\lambda}\big|b_{xx}(\phi_r^{(1)}+\theta\kappa_r,\phi_r^{(2)}+B_r^H)\big|d\theta d\lambda \big|(\kappa_s)^2-(\kappa_r)^2\big|+\int_{0}^{1}\int_{0}^{1}\big|b_{xy}(\phi_r^{(1)},\phi_r^{(2)}+\pi B_r^H)\big|d\pi d\lambda \big|B_s^H\kappa_s-B_r^H\kappa_r\big|\nonumber\\&+\int_{0}^{1}\int_{0}^{\lambda}\big|b_{yy}(\phi_r^{(1)},\phi_r^{(2)}+\nu B_r^H)\big|d\nu d\mu \big|(B^H_s)^2-(B^H_r)^2\big|\nonumber\\&\leq\nonumber L_5\Big(\frac{1}{2}|\phi_s-\phi_r|+\frac{1}{2}|B_s^2-B_r^2|+\frac{1}{6}|\kappa_s^2-\kappa_r^2|\Big)(\kappa_s)^2+L_6\Big(|\phi_s-\phi_r|+\frac{1}{2}|B_s^2-B_r^2|\Big)|B_s^H\kappa_s|\nonumber\\&+L_7\Big(\frac{1}{2}|\phi_s-\phi_r|+\frac{1}{6}|B_s^2-B_r^2|+\frac{1}{2}|\kappa_s^2-\kappa_r^2|\Big)(B_s^H)^2\nonumber\\&+\frac{1}{2}\|b_{xx}\|_{\infty}\big|(\kappa_s)^2-(\kappa_r)^2\big|+\|b_{xy}\|_{\infty}\big|B_s^H\kappa_s-B_r^H\kappa_r\big|+\frac{1}{2}\|b_{yy}\|_{\infty}\big|(B^H_s)^2-(B^H_r)^2\big|.\nonumber
\end{align}
Combining \eqref{the bound of remainder term} and \eqref{integral transform}, we have
\begin{align}
	&\G(1-\a)\Big|s^{\a}(D^{\a}_{0+}u^{-\a}R_u^{x,y})(s)\Big|\nonumber\\&=\Bigg|s^{-\a}R_s^{x,y}+\a s^{\a}\int_{0}^{s}\frac{s^{-\a}R_s^{x,y}-r^{-\a}R_r^{x,y}}{(s-r)^{\a+1}}dr\Bigg|\nonumber\\&\leq C(L_1)s^{-\a}\varepsilon^2+\a C(L_1) s^{\a}\varepsilon^2\int_{0}^{s}\frac{|s^{-\a}-r^{-\a}|}{(s-r)^{\a+1}}dr+\a s^{\a}\int_{0}^{s}\frac{r^{-\a}|R_s^{x,y}-R_r^{x,y}|}{(s-r)^{\a+1}}dr\nonumber\\&\leq C(L_1)s^{-\a}\varepsilon^2+\a C(L_1) s^{\a}\varepsilon^2\int_{0}^{s}\frac{|s^{-\a}-r^{-\a}|}{(s-r)^{\a+1}}dr\nonumber\\&+\a s^{\a}L_5(\kappa_s)^2\int_{0}^{s}\frac{r^{-\a}\Big(\frac{1}{2}|\phi_s-\phi_r|+\frac{1}{2}|B_s^2-B_r^2|+\frac{1}{6}|\kappa_s^2-\kappa_r^2|\Big)}{(s-r)^{\a+1}}dr\nonumber\\&+\a s^{\a}L_6|\kappa_s B_s^H|\int_{0}^{s}\frac{r^{-\a}\Big(|\phi_s-\phi_r|+\frac{1}{2}|B_s^2-B_r^2||\Big)}{(s-r)^{\a+1}}dr\nonumber\\&+\a s^{\a}L_7(B^H_s)^2\int_{0}^{s}\frac{r^{-\a}\Big(\frac{1}{2}|\phi_s-\phi_r|+\frac{1}{6}|B_s^2-B_r^2|+\frac{1}{2}|\kappa_s^2-\kappa_r^2|\Big)}{(s-r)^{\a+1}}dr\nonumber\\&+\frac{\|b_{xx}\|_{\infty}}{2}s^{\a}\int_{0}^{s}\frac{r^{-\a}\big|(\kappa_s)^2-(\kappa_r)^2\big|}{(s-r)^{\a+1}}dr+\|b_{xy}\|_{\infty}s^{\a}\int_{0}^{s}\frac{r^{-\a}\big|B_s^H\kappa_s-B_r^H\kappa_r\big|}{(s-r)^{\a+1}}dr\nonumber\\&+\frac{\|b_{yy}\|_{\infty}}{2}s^{\a}\int_{0}^{s}\frac{r^{-\a}\big|(B^H_s)^2-(B_r^H)^2\big|}{(s-r)^{\a+1}}dr\nonumber\\&\leq C(L_1)s^{-\a}\varepsilon^2+C(\a,L_1)s^{-\a}\varepsilon^2+\frac{\a}{2}s^{2\beta+\frac{1}{2}}L_5\varepsilon^2\int_{0}^{1}t^{-\a}(1-t)^{-\frac{1}{2}}dt\nonumber\\&+\frac{\a}{2}s^{3\beta-\a-1}L_5\varepsilon^3\int_{0}^{1}t^{-\a}(1-t)^{\beta-\a-1}dt+\frac{\a}{6}s^{3\beta-\a-1}L_5\varepsilon^3\int_{0}^{1}t^{-\a}(1-t)^{\beta-\a-1}dt\nonumber\\&+\a s^{2\beta+1}L_6\varepsilon^2\int_{0}^{1}t^{-\a}(1-t)^{-\frac{1}{2}}dt+\frac{\a}{2}s^{3\beta-\a-1}L_6\varepsilon^3\int_{0}^{1}t^{-\a}(1-t)^{\beta-\a-1}dt\nonumber\\&+\frac{\a}{2} s^{2\beta+1}L_7\varepsilon^2\int_{0}^{1}t^{-\a}(1-t)^{-\frac{1}{2}}dt+\frac{\a}{6}s^{3\beta-\a-1}L_7\varepsilon^3\int_{0}^{1}t^{-\a}(1-t)^{\beta-\a-1}dt\nonumber\\&+\frac{\a}{2}s^{3\beta-\a-1}L_7\varepsilon^3\int_{0}^{1}t^{-\a}(1-t)^{\beta-\a-1}dt+\frac{\|b_{xx}\|_{\infty}}{2}s^{2\beta-\a}\int_{0}^{1}t^{-\a}(1-t)^{2\beta-\a-1}dt\nonumber\\&+\|b_{xy}\|_{\infty}s^{\beta-\a}\varepsilon^2\int_{0}^{1}t^{-\a}(1-t)^{\beta-\a-1}dt+\|b_{xy}\|_{\infty}s^{2\beta-\a}\varepsilon^2\int_{0}^{1}t^{-\a}(1-t)^{\beta-\a-1}dt\nonumber\\&+\frac{\|b_{yy}\|_{\infty}}{2}s^{2\beta-\a}\int_{0}^{1}t^{-\a}(1-t)^{2\beta-\a-1}dt\nonumber\\&\leq C(\a,L_1)\varepsilon^2+C(\a,\b,L_5,L_6,L_7)s^{2\b+1}\varepsilon^2+C(\a,\b,L_5,L_6,L_7)s^{3\b-\a-1}\varepsilon^3\nonumber\\&+C(\a,\b,L_5,L_6,L_7)s^{2\b+1}\varepsilon^2+C(\a,\b,L_5,L_6,L_7)s^{2\b+1}\varepsilon^2+C(\a,\b)s^{2\b-\a}\varepsilon^2+C(\a,\b)s^{\b-\a}\varepsilon^2\nonumber
\end{align}
As a consequence,
\begin{align}
	\langle M \rangle_t=&c^2\int_{0}^{t}\Big(s^{\alpha}\Big(D^{\alpha}_{0+}u^{-\alpha}R_u^{x,y}\Big)(s)\Big)^2ds\leq C(\a,\b,L_1,L_5,L_6,L_7)\varepsilon^4.\nonumber
\end{align}
Applying the exponential inequality for martingales, we have
\begin{align}\label{estimate of expon}
	\mathbb{P}\Big(\Big|c\int_{0}^{t}s^{-\a}\Big(I^{\alpha}_{0+}u^{\alpha}R_u^{x,y}\Big)(s)dW_s\Big|>\xi, \|B^H\|_{\b}\leq \varepsilon\Big)\leq \exp\Big(-\frac{\xi^2}{2C(\a,\b,L_1,L_5,L_6,L_7)\varepsilon^{4}}\Big),
\end{align}
for every real number $c$.\\
Combining Lemma \ref{small ball of holder} and inequality \eqref{estimate of expon}, we see that
\begin{align}
	\mathbb{P}&\Big(\Big|c\int_{0}^{1}s^{-\a}\Big(I^{\alpha}_{0+}u^{\alpha}R_u^{x,y}\Big)(s)dW_s\Big|>\xi,\Big|\|B^H\|_{\b}\leq \varepsilon\Big)\nonumber\\&\leq \exp\Big(-\frac{\xi^2}{2C(\a,\b,L_1,L_5,L_6,L_7)\varepsilon^{4}}\Big)\exp\big(C_H \varepsilon^{-\frac{1}{H-\b}}\big).\nonumber
\end{align}
\\
Using the latter estimate we have for every $\delta>0$ and every $0<\varepsilon<1$
\begin{align}
	E&\Big(\exp(cI_{14})|\|B^H\|_{\beta}\leq \varepsilon\Big)\nonumber\\&\leq e^{\delta}+\int_{\delta}^{\infty}\exp\Big\{\xi-\frac{\xi^2}{2C(\a,\b,L_1,L_5,L_6,L_7)\varepsilon^{4}}+C_H \varepsilon^{-\frac{1}{H-\b}}\Big\}d \xi \nonumber\\&+\exp\Big\{\delta-\frac{\delta^2}{2C(\a,\b,L_1,L_5,L_6,L_7)\varepsilon^{4}}+C_H \varepsilon^{-\frac{1}{H-\b}}\Big\}.\nonumber
\end{align}
Letting $\varepsilon$ and then $\delta$ tend to zero, we obtain
\begin{align}\label{J_{14}}
	\limsup _{\varepsilon \rightarrow 0}\ E(\exp(cJ_{14})|\|B^H\|_{\b}<\varepsilon)\leq1,
\end{align}
for every real number $c$. \\

In conclusion, the following expression is a consequence of Lemma \ref{Separation lemma} and inequalities \eqref{small all 2}, \eqref{J_2}, \eqref{I_3}, \eqref{I_4}, \eqref{the part of J_1}, \eqref{J_{11}}, \eqref{I_{12}}, \eqref{I_{13}} and \eqref{I_{14}}.
$$
\lim _{\varepsilon \rightarrow 0} \frac{\mathbb{P}(\|Y-\phi^{(2)}\|_{\b}<\varepsilon)}{\mathbb{P}(\|B^H\|_{\b}<\varepsilon)}=\exp \left(-\frac{1}{2}\int_{0}^{1}\big|\dot{\phi}^2_s-s^{\a}\big(D^{\a}_{0+}u^{-\a}b(\phi_u)\big)(s)\big|^2ds-\frac{d_H}{2}\int_{0}^{1}b_y(\phi_{s})ds\right).
$$
The proof of Theorem \ref{th:regular case} is complete. \qed
\subsection{Proof of Theorem \ref{EL}}
we first derive Euler-Lagrange fractional equations for the non-degenerate case. When $H<\frac{1}{2}$, recall the expression of OM action functional proved in \cite{MN02}.

Consider the functional $I$ given by
$$I(\phi)=\frac{1}{2}\int_{0}^{1}\big|\dot{\phi}_s-s^{-\a}\big(I^{\a}_{0+}u^{\a}b(\phi_u)\big)(s)\big|^2ds+\frac{d_H}{2}\int_{0}^{1}b'(\phi_{s})ds.$$
Then we have
\begin{align}
	I(\phi+\varepsilon\psi)=\frac{1}{2}\int_{0}^{1}\Big|\dot{\phi}_s+\varepsilon\dot{\psi}_s-s^{-\a}\big(I^{\a}_{0+}u^{\a}b(\phi_u+\varepsilon\psi_u)\big)(s)\Big|^2ds+\frac{d_H}{2}\int_{0}^{1}b'(\phi_{s}+\varepsilon\psi_s)ds\nonumber.
\end{align}
Therefore, the derivative of $I(\phi+\varepsilon\psi)$ w.r.t. $\varepsilon$ equals
\begin{align}
	\frac{d}{d\varepsilon}I(\phi+\varepsilon\psi)&=\int_{0}^{1}\Big(\dot{\phi}_s+\varepsilon\dot{\psi}_s-s^{-\a}\big(I^{\a}_{0+}u^{\a}b(\phi_u+\varepsilon\psi_u)\big)(s)\Big)\big(\dot{\psi}_s-s^{-\a}\big(I^{\a}_{0+}u^{\a}b'(\phi_u+\varepsilon\psi_u)(\psi_u)\big)(s)\big)ds\nonumber\\&+\frac{d_H}{2}\int_{0}^{1}b''(\phi_{s}+\varepsilon\psi_s)\psi_sds.\nonumber
\end{align}
Let $\varepsilon=0$ and since $\phi$ is a minimizer of $I(\phi+\varepsilon\psi)$, we have
\begin{align}
	0&=\frac{d}{d\varepsilon}I(\phi+\varepsilon\psi)\Big\vert_{\varepsilon=0}\nonumber\\&=\int_{0}^{1}\Big(\dot{\phi}_s-s^{-\a}\big(I^{\a}_{0+}u^{\a}b(\phi_u)\big)(s)\Big)\big(\dot{\psi}_s-s^{-\a}\big(I^{\a}_{0+}u^{\a}b'(\phi_u)(\psi_u)\big)(s)\big)ds\nonumber\\&+\frac{d_H}{2}\int_{0}^{1}b''(\phi_{s})\psi_sds.\nonumber
\end{align}
It follows from integration by parts for fractional derivatives that
\begin{align}
	0&=\int_{0}^{1}\Big[-\frac{d}{dt}\big(\dot{\phi}_s-s^{-\a}\big(I^{\a}_{0+}u^{\a}b(\phi_u)\big)(s)\big)+\nonumber\\&+I^{\a}_{1-}\big(u^{-2\a}(I^{\a}_{0+})v^{\a}b(\phi_v)(u)\big)s^{\a}b'(\phi_s)+\frac{d_H}{2}b''(\phi_{s})\Big]\psi_sds\nonumber.
\end{align}
Finally, due to $\psi$ has compact support, we obtain the
fractional Euler-Lagrange equation
$$I^{\a}_{1-}\big(u^{-2\a}(I^{\a}_{0+})v^{\a}b(\phi_v)(u)\big)s^{\a}b'(\phi_s)+\frac{d_H}{2}b''(\phi_{s})=\frac{d}{dt}\big(\dot{\phi}_s-s^{-\a}\big(I^{\a}_{0+}u^{\a}b(\phi_u)\big)(s)\big).$$
Similarly, we can obtain a fractional Euler-Lagrange equation when $H>\frac{1}{2}$.
$$D^{\a}_{1-}\big(u^{-2\a}(D^{\a}_{0+})v^{\a}b(\phi_v)(u)\big)s^{\a}b'(\phi_s)+\frac{d_H}{2}b''(\phi_{s})=\frac{d}{dt}\big(\dot{\phi}_s-s^{-\a}\big(D^{\a}_{0+}u^{\a}b(\phi_u)\big)(s)\big).$$
Next, let us turn our attention to the degenerate case. Consider the functional given by
$$I(\phi^{(1)})=\frac{1}{2}\int_{0}^{1}\big|\ddot{\phi}^{(1)}_s-s^{\a}\big(D^{\a}_{0+}u^{-\a}b(\phi^{(1)}_u,\dot{\phi}^{(1)}_u)\big)(s)\big|^2ds+\frac{d_H}{2}\int_{0}^{1}b_y(\phi^{(1)}_s,\dot{\phi}^{(1)}_{s})ds,$$
where $H<\frac{1}{2}$. Then we have
\begin{align}
	I(\phi^{(1)}+\varepsilon\psi^{(1)})&=\frac{1}{2}\int_{0}^{1}\Big|\ddot{\phi}^{(1)}_s+\varepsilon\ddot{\psi}^{(1)}_s-s^{-\a}\big(I^{\a}_{0+}u^{\a}b(\phi^{(1)}_u+\varepsilon\psi^{(1)}_u,\dot{\phi}^{(1)}_u+\varepsilon\dot{\psi}^{(1)}_u)\big)(s)\Big|^2ds\nonumber\\&+\frac{d_H}{2}\int_{0}^{1}b_y(\phi^{(1)}_{s}+\varepsilon\psi^{(1)}_s,\dot{\phi}^{(1)}_s+\varepsilon\dot{\psi}^{(1)}_s)ds\nonumber.
\end{align}
Therefore, the derivative of $I(\phi^{(1)}+\varepsilon\psi^{(1)})$ w.r.t. $\varepsilon$ equals
\begin{align}
	&\frac{d}{d\varepsilon}I(\phi^{(1)}+\varepsilon\psi^{(1)})\nonumber\\&=\int_{0}^{1}\Big(\ddot{\phi}^{(1)}_s+\varepsilon\ddot{\psi}^{(1)}_s-s^{-\a}\big(I^{\a}_{0+}u^{\a}b(\phi^{(1)}_u+\varepsilon\psi^{(1)}_u,\dot{\phi}^{(1)}_u+\varepsilon\dot{\psi}^{(1)}_u)\big)(s)\Big)\ddot{\psi}_s^{(1)}\nonumber\\&-\Big(\ddot{\phi}^{(1)}_s+\varepsilon\ddot{\psi}^{(1)}_s-s^{-\a}\big(I^{\a}_{0+}u^{\a}b(\phi^{(1)}_u+\varepsilon\psi^{(1)}_u,\dot{\phi}^{(1)}_u+\varepsilon\dot{\psi}^{(1)}_u)\big)(s)\Big)s^{-
		\a}\big(I^{\a}_{0+}u^{\a}b_x(\phi^{(1)}_u+\varepsilon\psi^{(1)}_u,\dot{\phi}^{(1)}_u+\varepsilon\dot{\psi}^{(1)}_u)\psi^{(1)}_u\big)(s)\nonumber\\&-\Big(\ddot{\phi}^{(1)}_s+\varepsilon\ddot{\psi}^{(1)}_s-s^{-\a}\big(I^{\a}_{0+}u^{\a}b(\phi^{(1)}_u+\varepsilon\psi^{(1)}_u,\dot{\phi}^{(1)}_u+\varepsilon\dot{\psi}^{(1)}_u)\big)(s)\Big)s^{-
		\a}\big(I^{\a}_{0+}u^{\a}b_y(\phi^{(1)}_u+\varepsilon\psi^{(1)}_u,\dot{\phi}^{(1)}_u+\varepsilon\dot{\psi}^{(1)}_u)\dot{\psi}^{(1)}_u\big)(s)\nonumber\\&+\frac{d_H}{2}\int_{0}^{1}b_{yx}(\phi^{(1)}_{s}+\varepsilon\psi^{(1)}_s,\dot{\phi}^{(1)}_s+\varepsilon\dot{\psi}^{(1)}_s)\psi^{(1)}_sds+\frac{d_H}{2}\int_{0}^{1}b_{yy}(\phi^{(1)}_{s}+\varepsilon\psi^{(1)}_s,\dot{\phi}^{(1)}_s+\varepsilon\dot{\psi}^{(1)}_s)\dot{\psi}^{(1)}_sds.\nonumber
\end{align}
Let $\varepsilon=0$ and since $\phi^{(1)}$ is a minimizer of $I(\phi^{(1)}+\varepsilon\psi^{(1)})$, we have
\begin{align}
	0&=\frac{d}{d\varepsilon}I(\phi^{(1)}+\varepsilon\psi^{(1)})\Big\vert_{\varepsilon=0}\nonumber\\&=\int_{0}^{1}\Big(\ddot{\phi}^{(1)}_s-s^{-\a}\big(I^{\a}_{0+}u^{\a}b(\phi^{(1)}_u,\dot{\phi}^{(1)}_u\big)(s)\Big)\ddot{\psi}_s^{(1)}\nonumber\\&-\Big(\ddot{\phi}^{(1)}_s-s^{-\a}\big(I^{\a}_{0+}u^{\a}b(\phi^{(1)}_u,\dot{\phi}^{(1)}_u)\big)(s)\Big)s^{-
		\a}\big(I^{\a}_{0+}u^{\a}b_x(\phi^{(1)}_u,\dot{\phi}^{(1)}_u)\psi^{(1)}_u\big)(s)\nonumber\\&-\Big(\ddot{\phi}^{(1)}_s-s^{-\a}\big(I^{\a}_{0+}u^{\a}b(\phi^{(1)}_u,\dot{\phi}^{(1)}_u)\big)(s)\Big)s^{-
		\a}\big(I^{\a}_{0+}u^{\a}b_y(\phi^{(1)}_u,\dot{\phi}^{(1)}_u)\dot{\psi}^{(1)}_u\big)(s)\nonumber\\&+\frac{d_H}{2}\int_{0}^{1}b_{yx}(\phi^{(1)}_{s},\dot{\phi}^{(1)}_s)\psi^{(1)}_sds+\frac{d_H}{2}\int_{0}^{1}b_{yy}(\phi^{(1)}_{s},\dot{\phi}^{(1)}_s)\dot{\psi}^{(1)}_sds.\nonumber
\end{align}
It follows from integration by parts for fractional derivatives that
\begin{align}
	0&=\int_{0}^{1}\Bigg[\frac{d^2}{dt^2}\big(\ddot{\phi}^{(1)}_s-s^{\a}b_x(\phi_s^{(1)},\dot{\phi}^{(1)}_s)\Big(I^{\a}_{1-}u^{\a}\Big(\ddot{\phi}^{(1)}_u-u^{-\a}\big(I^{\a}_{0+}v^{\a}b(\phi^{(1)}_v,\dot{\phi}^{(1)}_v\big)(u)\Big)\Big)(s)\nonumber\\&+\frac{d}{dt}\Big[s^{\a}b_y(\phi_s^{(1)},\dot{\phi}^{(1)}_s)\Big(I^{\a}_{1-}u^{\a}\Big(\ddot{\phi}^{(1)}_u-u^{-\a}\big(I^{\a}_{0+}v^{\a}b(\phi^{(1)}_v,\dot{\phi}^{(1)}_v\big)(u)\Big)\Big)(s)\Big]\nonumber\\&+\frac{d_H}{2}b_{yx}(\phi^{(1)}_{s},\dot{\phi}^{(1)}_s)-\frac{d_H}{2}\frac{d}{dt}\Big(b_{yy}(\phi^{(1)}_{s},\dot{\phi}^{(1)}_s)\Big)\Bigg]\psi^{(1)}_sds.\nonumber
\end{align}
Finally, due to $\psi$ has compact support, we obtain the
fractional Euler-Lagrange equation
\begin{align}
	0&=\frac{d^2}{dt^2}\big(\ddot{\phi}^{(1)}_s-s^{\a}b_x(\phi_s^{(1)},\dot{\phi}^{(1)}_s)\Big(I^{\a}_{1-}u^{\a}\Big(\ddot{\phi}^{(1)}_u-u^{-\a}\big(I^{\a}_{0+}v^{\a}b(\phi^{(1)}_v,\dot{\phi}^{(1)}_v\big)(u)\Big)\Big)(s)\nonumber\\&+\frac{d}{dt}\Big[s^{\a}b_y(\phi_s^{(1)},\dot{\phi}^{(1)}_s)\Big(I^{\a}_{1-}u^{\a}\Big(\ddot{\phi}^{(1)}_u-u^{-\a}\big(I^{\a}_{0+}v^{\a}b(\phi^{(1)}_v,\dot{\phi}^{(1)}_v\big)(u)\Big)\Big)(s)\Big]\nonumber\\&+\frac{d_H}{2}b_{yx}(\phi^{(1)}_{s},\dot{\phi}^{(1)}_s)-\frac{d_H}{2}\frac{d}{dt}\Big(b_{yy}(\phi^{(1)}_{s},\dot{\phi}^{(1)}_s)\Big).\nonumber
\end{align}
Similarly, we can obtain a fractional Euler-Lagrange equation when $H>\frac{1}{2}$.
\begin{align}
	0&=\frac{d^2}{dt^2}\big(\ddot{\phi}^{(1)}_s-s^{\a}b_x(\phi_s^{(1)},\dot{\phi}^{(1)}_s)\Big(D^{\a}_{1-}u^{\a}\Big(\ddot{\phi}^{(1)}_u-u^{-\a}\big(D^{\a}_{0+}v^{\a}b(\phi^{(1)}_v,\dot{\phi}^{(1)}_v\big)(u)\Big)\Big)(s)\nonumber\\&+\frac{d}{dt}\Big[s^{\a}b_y(\phi_s^{(1)},\dot{\phi}^{(1)}_s)\Big(D^{\a}_{1-}u^{\a}\Big(\ddot{\phi}^{(1)}_u-u^{-\a}\big(D^{\a}_{0+}v^{\a}b(\phi^{(1)}_v,\dot{\phi}^{(1)}_v\big)(u)\Big)\Big)(s)\Big]\nonumber\\&+\frac{d_H}{2}b_{yx}(\phi^{(1)}_{s},\dot{\phi}^{(1)}_s)-\frac{d_H}{2}\frac{d}{dt}\Big(b_{yy}(\phi^{(1)}_{s},\dot{\phi}^{(1)}_s)\Big).\nonumber
\end{align}
\section{Discussion}
In this work, we have derived the Onsager–Machlup action function for  degenerate stochastic differential equations driven by fractional Brownian motion. With this function as a Lagrangian, we obtained two classes of fractional Euler-Lagrange equations, the one is based the results of \cite{MN02}, and the other is based Theorem \ref{th:singular case} and  Theorem \ref{th:regular case}. A point that should be stressed is that for the (degenerate) stochastic differential equations driven by fractional noise, we cannot extend the results to higher dimensional systems. Moreover, we will also study the mathematical properties of fractional differential equations obtained in Theorem \ref{EL} in the future.
\newcommand{\etalchar}[1]{$^{#1}$}
\providecommand{\bysame}{\leavevmode\hbox to3em{\hrulefill}\thinspace}
\providecommand{\MR}{\relax\ifhmode\unskip\space\fi MR }
\providecommand{\MRhref}[2]{%
	\href{http://www.ams.org/mathscinet-getitem?mr=#1}{#2}
}
\providecommand{\href}[2]{#2}

\section*{Acknowledgements}
The work is supported in part by the NSFC Grant Nos. 12171084 and  the fundamental Research
Funds for the Central Universities No. RF1028623037.

\providecommand{\bysame}{\leavevmode\hbox to3em{\hrulefill}\thinspace}
\providecommand{\MR}{\relax\ifhmode\unskip\space\fi MR }
\providecommand{\MRhref}[2]{%
	\href{http://www.ams.org/mathscinet-getitem?mr=#1}{#2}
}
\providecommand{\href}[2]{#2}

\section*{Conflicts of interests}
The authors declare no conflict of interests.

\end{document}